\documentclass[12pt]{article}

\usepackage{amsmath,amsthm,amssymb}
\allowdisplaybreaks
\usepackage{mathrsfs}
\usepackage{cite}

\usepackage[pdftex,bookmarksopen=true,bookmarks=true,unicode,setpagesize]{hyperref}
\hypersetup{colorlinks=true,linkcolor=black,citecolor=black}

\usepackage{color}

\newtheorem{theorem}{Theorem}
\newtheorem{corollary}[theorem]{Corollary}
\newtheorem{lemma}[theorem]{Lemma}

\theoremstyle{remark}

\theoremstyle{remark}

\theoremstyle{remark}
\newtheorem{remark}[theorem]{Remark}

\newcommand{\R}{{\mathbb R}}

\newcommand{\set}{\mathscr V}

\begin{document}

\begin{center}{\Large \bf
 A moment problem for random  discrete measures
}\end{center}

{\large Yuri G. Kondratiev}\\
 Fakult\"at f\"ur Mathematik, Universit\"at
Bielefeld, Postfach 10 01 31, D-33501 Bielefeld, Germany;  NPU, Kyiv, Ukraine\\
 e-mail:
\texttt{kondrat@mathematik.uni-bielefeld.de}\vspace{2mm}

{\large Tobias Kuna}\\ University of Reading,
Department of  Mathematics,
Whiteknights,
PO Box 220,
Reading RG6 6AX, U.K.\\
 e-mail:
\texttt{ t.kuna@reading.ac.uk}\vspace{2mm}

{\large Eugene Lytvynov}\\ Department of Mathematics,
Swansea University, Singleton Park, Swansea SA2 8PP, U.K.\\
e-mail: \texttt{e.lytvynov@swansea.ac.uk}\vspace{6mm}



{\small

\begin{center}
{\bf Abstract}
\end{center}
\noindent

Let $X$ be a locally compact Polish space. A random measure  on $X$ is a probability measure on the space of all (nonnegative) Radon measures on $X$.
Denote by  $\mathbb K(X)$ the cone   of all  Radon measures $\eta$ on $X$ which are of  the form $\eta=\sum_{i}s_i\delta_{x_i}$,
where, for each $i$, $s_i>0$ and  $\delta_{x_i}$ is the Dirac measure at $x_i\in X$.
A random discrete measure  on $X$ is a  probability measure on $\mathbb K(X)$.
The main result of the  paper states a necessary and sufficient condition (conditional upon a mild {\it a priori\/} bound) when a random measure $\mu$ is also a random discrete measure.
 This condition is formulated solely  in  terms of moments  of the random measure $\mu$.
Classical examples of random discrete measures are completely random measures and additive subordinators, however, the main result holds independently of any independence property. As a corollary, a characterisation via a moments is given when a random measure is a point process.
 } \vspace{2mm}

\section{Introduction}
\label{tuyr7r}

Let $X$ be a locally compact Polish space, and let $\mathcal B(X)$ denote the associated Borel $\sigma$-algebra.
For example, $X$ can be the Euclidean space $\R^d$, $d\in\mathbb N$.
Let  $\mathbb M(X)$ denote the space of all (nonnegative) Radon measures on $(X,\mathcal B(X))$. The space $\mathbb M(X)$ is equipped with the vague topology.
Let $\mathcal B(\mathbb M(X))$ denote the Borel $\sigma$-algebra on $\mathbb M(X)$.

  Let us define the  {\it cone of (nonnegative) discrete Radon measures on $X$}   by
$$\mathbb K(X):=\left\{
\eta=\sum_i s_i\delta_{x_i}\in \mathbb M(X) \,\Big|\, s_i>0,\, x_i\in X
\right\}.$$
 Here $\delta_{x_i}$ denotes the Dirac measure with mass at $x_i$.
 In the above representation, the atoms $x_i$ are assumed to be distinct, i.e.,  $x_i \neq x_j$ for $i \neq j$, and their total number is at most countable. By convention, the cone $\mathbb K(X)$ contains the null mass $\eta=0$, which is represented by the sum over an empty set of indices $i$.
  As shown in \cite{HKPR}, $\mathbb K(X)\in\mathcal B(\mathbb M(X))$.

 A {\it random measure  on $X$} is a
 measurable mapping $\xi:\Omega\to\mathbb M(X)$, where  $(\Omega,\mathcal F,P)$ is a probability space,
see e.g.\ \cite{DVJ1,DVJ2,Kallenberg}. A random measure which takes values in $\mathbb K(X)$ with probability one will be called a {\it random discrete measure.} We will give results which characterize when a random measure is a random discrete measure in terms of its moments.

Let us recall  the classical characterization of a completely random measure by Kingman \cite{Kingman, DVJ2}.
A random measure $\xi$  is called {\it completely random} if, for any mutually
 disjoint sets $A_1,\dots,A_n\in\mathcal B(X)$, the random variables $\xi(A_1),\dots,\xi(A_n)$ are independent. Kingman's  theorem states that every completely random measure $\xi$ can be represented as $\xi=\xi_d+\xi_f+\xi_r$.
 Here $\xi_d$, $\xi_f$, $\xi_r$ are independent completely random measures such that: $\xi_d$ is a deterministic measure on $X$ without atoms; $\xi_f$ is a random measure with fixed (non-random) atoms, that is there exists a deterministic countable collection of points $\{x_i\}$ in $X$ and non-negative independent random numbers $\{a_i\}$ with $\xi_f=\sum_i a_i\delta_{x_i}$; finally the most essential part $\xi_r$ is an extended marked Poisson process which has no fixed atoms, in particular with probability one $\xi_r$ is of the form $\sum_{j}b_j\delta_{y_j}$,
 where $\{b_j\}$ are non-negative random numbers and $\{y_j\}$ are random points in $X$.

Thus, by Kingman's result a completely random measure is a random discrete measure up to a non-random component. If one drops the assumption that the random measure is completely random, one cannot expect anymore to concretely characterize the distribution of $\xi$. Thus, a natural appropriate question  is to ask when a random measure is a random discrete measure. One may be tempted to replace the assumption of complete randomness by a property of a sufficiently strong decay of correlation. However, the result of this paper shows that such an  assumption cannot be sufficient.

Note that, in most interesting examples of completely random measures, the set of atoms of the random discrete measure is almost surely  dense in  $X$. A  study of countable dense random subsets of $X$ leads to  ``situations in which probabilistic statements about such sets can be uninformative'' \cite{Kendall}, see also \cite{AB}. It is the presence of the weights $s_i$ in the definition of a random discrete measure that makes a real difference.

 An important characteristic of a random measure is its moment sequence.   We say that a random measure $\xi$ has finite moments of all orders if, for each $n\in\mathbb N$ and all bounded subsets $A \in \mathcal{B}(X)$,
 \[
 \mathbb{E}[\xi(A)^n] < \infty .
 \]
Then, the {\it $n$-th moment measure of $\xi$} is the unique symmetric measure $M^{(n)} \in\mathbb{M}(X^n)$ defined by the following relation
\[
\forall A_1, \ldots A_n \in \mathcal{B}(X) : \qquad M^{(n)}(A_1 \times \dots \times A_n) :=\mathbb{E}[\xi(A_1) \dotsm \xi(A_n)].
\]
 We also set $M^{(0)}:=\mathbb E(1)=1$. The $(M^{(n)})_{n=0}^\infty$ is called the {\it moment sequence of the random measure $\xi$.}

The main result  of this paper is a solution of the following problem: {\it Assume that $\xi$ is a random measure on $X$ whose moment sequence $(M^{(n)})_{n=0}^\infty$ is known and satisfies a mild {\it a priori\/} bound. Give a necessary and sufficient condition, in terms of the moments $(M^{(n)})_{n=0}^\infty$\,, for  $\xi$ to be   a random discrete measure, i.e., for the distribution of the random measure $\xi$ to be  concentrated on $\mathbb K(X)$.}

As a consequence of our main result we also obtain a solution of the (infinite dimensional) moment problem on $\mathbb K(X)$. Since we will only use the distribution of
 a random measure on $\mathbb M(X)$, in what follows, under a random measure we will always understand a probability measure on $\mathbb M(X)$, and under a random discrete measure a probability measure concentrated on the subset $\mathbb K(X)$.

In Section~\ref{secmoment} we state three corollaries of the main result. In Corollary~\ref{sec4cor1}, 
{\it we give  a necessary and sufficient condition for a sequence of Radon measures, $(M^{(n)})_{n=0}^\infty$\,, to be the moment sequence of a random discrete measure}, cf.\ Section~\ref{secmoment} for details.
In Corollary~\ref{hfyrd6edhgyhg}, {\it we give a necessary and sufficient condition, in terms of the moments $(M^{(n)})_{n=0}^\infty$, for a random measure  to be a simple point process.} In Corollary~\ref{huigtt}, we relate the previous corollary to the analogous result in terms of the so-called  {\it generalized correlation functions}.
(These  results are also conditional upon an {\it a priori\/} bound satisfied by $(M^{(n)})_{n=0}^\infty$.)
 
 Our main result is very different in spirit and technique to the known results about the localization of measures on cones. As far as we know, all known techniques require the cone under consideration to be closed, cf.\ \cite{S}, but  $\mathbb K(X)$ is dense in  $\mathbb M(X)$, cf.\ the proof of separability in Proposition~A2.5.III in \cite{DVJ1}.

In order to describe the main result more precisely we have to introduce some further notation. Let $i_1,\ldots i_k \in \mathbb{N}$ with $i_1 + \dots +i_k=n$. Denote by $M_{i_1, \ldots , i_k}$ the restriction of $M^{(n)}$ to the following subset of $X^n$
\[
\left\{ \big(\underbrace{x_1, \ldots, x_1}_{i_1}, x_2, \ldots , \underbrace{x_k, \ldots , x_k}_{i_k}\,\big) \in X^n \,\Big|\, x_i \neq x_j \mbox{ for } i \neq j\right\}.
\]
Denote by $X^{(k)}_{\widehat0}$ the collection of points $(x_1,\dots,x_k)\in X^k$ whose coordinates are all  different.  We consider $M_{i_1, \ldots , i_k}$ as a measure on $X^{(k)}_{\widehat0}$, cf.\ Section~\ref{5tew5w} for details.

It is clear that a result of the type we wish to derive can only hold under an appropriate estimate on the growth of the measures $M^{(n)}$. Below we will assume that the following conditions are satisfied, see also Remark~\ref{rmc2}:

\begin{itemize}
\item[(C1)] For each $\Lambda\in\mathcal B_c(X)$, there exists a constant $C_\Lambda>0$ such that
\begin{equation}\label{hydyrk}
M^{(n)}(\Lambda^n)\le C_\Lambda^n\,n!\,,\quad n\in\mathbb N.\end{equation}
\end{itemize}
Here $\mathcal B_c(X)$ denotes the collection of all sets from $\mathcal B(X)$ which have compact closure.
\begin{itemize}
\item[(C2)] For each $\Lambda\in\mathcal B_c(X)$, there exists a constant $C_\Lambda'>0$ such that
\begin{equation}\label{lhgigikbgi}
M^{(n)}(\Lambda^{(n)}_{\widehat0})\le (C_\Lambda')^nn!\,,\quad\forall n \in \mathbb{N}
\end{equation}
and for any sequence $\{\Lambda_k\}_{k=1}^\infty\in\mathcal B_c(X)$ such that $\Lambda_k\downarrow\varnothing$, we have $C_{\Lambda_k}'\to0$ as $k\to\infty$.
\end{itemize}

We fix a sequence $(\Lambda_l)_{l=1}^\infty$ of compact subsets of $X$ such that $\Lambda_1\subset\Lambda_2\subset\Lambda_3\subset\dotsm$ and $\bigcup_{l=1}^\infty\Lambda_l=X$. For example, in the case $X=\R^d$, one may choose $\Lambda_l=[-l,l]^d$.

\begin{theorem}\label{jkgfrt7urd}
Let $\mu$ be a random measure on $X$, i.e., a probability measure on\linebreak $(\mathbb M(X),\mathcal B(\mathbb M(X)))$. Assume that $\mu$ has finite moments,
and let $(M^{(n)})_{n=0}^\infty$ be its moment sequence.
Further assume that conditions {\rm (C1)} and {\rm (C2)} are
satisfied.
Then $\mu$ is a random discrete measure, i.e., $\mu(\mathbb K(X))=1$, if and only if the moment sequence $(M^{(n)})_{n=0}^\infty$ satisfies the following conditions:

\begin{itemize}

\item[(i)] For any $n\in\mathbb N$, $\Delta\in\mathcal B_c(X^{(n)}_{\widehat0})$, and $\mathbf i=(i_1,\dots,i_n)\in\mathbb Z_+^n$,
let \begin{equation}\xi^\Delta_{\mathbf i}=\xi^\Delta_{i_1,\dots,i_n}:=\frac1{n!}\,M_{i_1+1,\dots,i_n+1}(\Delta). \label{ude7re7uvt}\end{equation} cf.\ for more details \eqref{ude7re7u}. Here $\mathbb Z_+:=\mathbb N\cup\{0\}$.

Then the sequence
 $(\xi^\Delta_{\mathbf i})_{\mathbf i\in\mathbb Z_+^n}$ is positive definite, i.e., for $N\in\mathbb N$ and any finite sequence of complex numbers indexed by elements of $\mathbb Z_+^n$,  $(z_{\mathbf i})_{\mathbf i\in\mathbb Z_+^n,\, |\mathbf i|\le N}$, we have
 $$\sum_{\substack{i_1, \ldots , i_n = 1\\
 j_1, \ldots , j_n = 1}}^N \xi^\Delta_{i_1+j_1, \ldots , i_n+j_n}\, z_{i_1,\dots,i_n}\overline{z_{j_1,\dots,j_n}}\ge0.$$
Here $|\mathbf i|:=\max\{i_1,\dots,i_n\}$.

 \item[(ii)]
   For each $\Delta\in\mathcal B_c(X^{(n)}_{\widehat0})$ of the form $\Delta=(\Lambda_l)^{(n)}_{\widehat 0}$ with $l\in\mathbb N$, set
 \begin{equation}\label{yufu8r86orftf}
 r_i^\Delta:=\xi^\Delta_{i,0,0,\dots,0}\,,\quad i\in\mathbb Z_+.\end{equation}
  Then, for any finite sequence of complex numbers, $(z_n)_{n=0}^N$, we have
     \begin{equation}\label{bkjgtfi8ytf8gi}\sum_{i,j=0}^N r^\Delta_{i+j+1}\, z_i\,\overline{z_j}\ge0,
\end{equation}
and furthermore
\begin{equation}\label{ug86r8o7}\sum_{k=1}^\infty (D^\Delta_{k-1}D^\Delta_k)^{-1}\operatorname{det}\left[\begin{matrix}
r^\Delta_1&r^\Delta_2&\dots&r^\Delta_k\\
r^\Delta_2&r^\Delta_3&\dots&r^\Delta_{k+1}\\
\vdots&\vdots&\vdots&\vdots\\
r^\Delta_k&r^\Delta_{k+1}&\dots&r^\Delta_{2k-1}
\end{matrix}
\right]^2=\infty,
\end{equation}
where
$$ D_k:= \operatorname{det}\left[\begin{matrix}
r^\Delta_0&r^\Delta_1&\dots&r^\Delta_k\\
r^\Delta_1&r^\Delta_2&\dots&r^\Delta_{k+1}\\
\vdots&\vdots&\vdots&\vdots\\
r^\Delta_k&r^\Delta_{k+1}&\dots&r^\Delta_{2k}
\end{matrix}
\right],\quad k\in\mathbb Z_+.
$$
\end{itemize}

\end{theorem}

Let us now briefly describe the strategy we  follow in this paper.
Denote $ \R_+:=(0,\infty)$. We introduce a logarithmic metric on $ \R_+$: for $a,b\in \R_+$,
$\operatorname{dist}(a,b):=\left|\ln\left(\frac ab\right)\right|$. Then $ \R_+$  becomes a locally compact Polish space. Thus, $Y:=X\times  \R_+$ is also a locally compact Polish space. We consider the configuration space $\Gamma(Y)$, i.e., the space of all locally finite subsets of $Y$. This space is also equipped with the vague topology. A {\it (simple) point process} in $Y$ is a probability measure on $(\Gamma(Y),\mathcal B(\Gamma(Y)))$.
All processes which we consider in this paper are simple, so any time we mention a point process, we will actually mean a simple point process.
A point process is  (uniquely) characterized by its {\it correlation measure\/}  (also called the {\it factorial moments measures\/}), see e.g.\ \cite{DVJ1}.

Let $\mu$ be a random discrete measure on $X$. It is often  convenient to interpret $\mu$ as a point process in $Y$. More precisely, take any discrete Radon measure $\eta=\sum_i s_i\delta_{x_i}\in\mathbb K(X)$ and set
$$\mathcal E\eta:=\sum_i \delta_{(x_i,s_i)}.$$
As easily seen $\mathcal E\eta\in \Gamma(Y)$. Furthermore, it can be shown that the mapping $\mathcal E:\mathbb K(X)\to\Gamma(Y)$ is measurable,
see \cite{HKPR}.
 (Note, however, that the range of the mapping $\mathcal E$ is not the whole space $\Gamma(Y)$, see the definition in equation (\ref{fdstrset6s}) below.) We denote $\nu:=\mathcal E(\mu)$, i.e., the pushforward of $\mu$ under $\mathcal E$. Thus, $\nu$ is a point process in $Y$.
Hence, one can study
  the random discrete measure $\mu$ through the point process $\nu$.

Our strategy to solve the main problem is to first construct the point process $\nu$ associated to the searched random discrete measure $\mu$. An important step along the way here is to solve the following problem, which is of independent interest in itself: {\it How can one recover the correlation measure   of the associated point process $\nu$ from  the moment sequence  $(M^{(n)})_{n=0}^\infty$ of a random discrete measure $\mu$?}
A solution to this problem is given in Section~\ref{5tew5w}. Our approach is significantly influenced  by the paper of Rota and Wallstrom \cite{RW}, which combines ideas of (stochastic) integration  and combinatorics. Additionally, to find the correlation measure of $\nu$ concretely, one has to solve a sequence of finite-dimensional moment problems.
A solution to the main problem is given in Section~\ref{hftuttfd7u} and the consequence for the moment problem on $\mathbb{K}(X)$ is discussed in Section~\ref{secmoment}.

Beside completely random measures or additive subordinators in the case $X=\mathbb{R}_+$ (in particular,  L\'evy processes which are subordinators),  we would like to mention   the gamma measure, the spatial version of the gamma process, see e.g.\ \cite{VGG1,VGG2,TsVY} for interesting properties. The gamma measure is the completely random discrete measure $\mu$ on $X=\R^d$ for which  $\mathcal E(\mu)=\nu$ is the Poisson point process  in $\R^d\times \R_+$ with intensity measure $dx\, s^{-1}e^{-s}\,ds$. Note that Gibbs perturbations of the gamma measure have been studied in \cite{HKPR}. These are also random discrete measures which have a.s.\ a dense set of atoms.

\section{Recovering the correlation measure of $\nu$}\label{5tew5w}

A partition of a nonempty set $Z$ is any finite collection $\pi=\{A_1,\dots,A_k\}$, where $A_1,\dots,A_k$ are mutually disjoint nonempty subsets of $Z$ such that $Z=\bigcup_{i=1}^k A_i$.
 The sets $A_1,\dots,A_k$ are called blocks of the partition $\pi$.

 For each $n\in\mathbb N$, denote by $\Pi(n)$ the set of all partitions of the set $\{1,2,\dots,n\}$. For each partition $\pi=\{A_1,\dots,A_k\}\in\Pi(n)$, we denote by $X^{(n)}_\pi$ the subset of $X^n$ which consists of all $(x_1,\dots,x_n)\in X^n$ such that, for any $1\le i<j\le n$, $x_i=x_j$ if and only if $i$ and $j$ belong to the same block of the partition  $\pi$, say $A_l$. For example, for the so-called zero partition
 $\widehat0=\big\{\{1\},\{2\},\dots,\{n\}\big\}$, the set $X^{(n)}_{\widehat0}$ consists of all points $(x_1,\dots,x_n)\in X^n$ whose coordinates are all  different. For the so-called one partition $\widehat 1=\big\{
 \{1,2,\dots,n\}\big\}$, the set $X^{(n)}_{\widehat1}$ consists of all points $(x_1,\dots,x_n)\in X^n$ such that $x_1=x_2=\dots=x_n$. Clearly, the collection of sets $X_\pi^{(n)}$ with $\pi$ running over $\Pi(n)$ forms a partition of $X^n$.

 Let  $m^{(n)}$ be any nonnegative Radon measure on $X^n$, i.e., $m^{(n)}\in\mathbb M(X^n)$. For each partition $\pi\in\Pi(n)$, we denote by $m^{(n)}_\pi$ the restriction of the measure $m^{(n)}$ to the set $X^{(n)}_\pi$.
Note that we may also consider $m_\pi^{(n)}$ as a measure on $X^n$ by setting $$m_\pi^{(n)}(X^n\setminus X^{(n)}_\pi):=0.$$
Then we get
$$ m^{(n)}=\sum_{\pi\in\Pi(n)}m^{(n)}_\pi.$$

Let us fix a  partition $\pi=\{A_1,A_2,\dots,A_k\}\in \Pi(n)$.
 Here and below, we will always assume   that the blocks of the partition are enumerated so that
$$\min A_1<\min A_2<\dots<\min A_k.$$
We  denote by $|\pi|$ the number of blocks in the partition $\pi$. We  construct a measurable, bijective mapping, for $k=|\pi|$,
$$ B_\pi: X_\pi^{(n)}\to X^{(k)}_{\widehat0}$$
as follows.
For any $(x_1,\dots,x_n)\in X_\pi^{(n)}$, we set
$$ B_\pi(x_1,\dots,x_n)=(y_1,\dots,y_{k}),$$
where, for $i=1,2,\dots,k$, $y_i=x_j$ for a $j\in A_i$ (recall that $x_j=x_{j'}$ for all $j,j' \in A_i$).
Note that, if $\pi=\widehat0$, then $B_\pi$ is just the identity mapping. We denote by $B_\pi (m_\pi^{(n)})$ the pushforward of the measure $m_\pi^{(n)}$ under $B_\pi$.

Let us now additionally assume that the initial measure $m^{(n)}$ is symmetric, i.e., the measure $m^{(n)}$ remains invariant under the natural action of permutations $\sigma\in\mathfrak S_n$ on $X^n$. (Here $\mathfrak S_n$ denotes the symmetric group of degree $n$.)
 For a partition $\pi$ as in the above paragraph, we set, for each $l=1,2,\dots,k$, $i_l:=|A_l|$, the number of elements of the block $A_l$. Note that
$i_1+i_2+\dots+i_k=n$.
Since $m^{(n)}$ is symmetric, it is clear that the measure  $B_\pi (m_\pi^{(n)})$ is completely identified by the numbers $i_1,\dots,i_k$. That is, if $\pi'=\{A_1',\dots,A_k'\}$ is another partition from $\Pi(n)$ with
$|A'_l|=i_l$, $l=1,\dots,k$,
 then
$B_\pi (m_\pi^{(n)})=B_{\pi'} (m_{\pi'}^{(n)})$.
Hence, we will denote
\begin{equation}\label{ytdey6es6y}m_{i_1,\dots,i_k}:=B_\pi (m_\pi^{(n)}),\end{equation}
and we may assume, without loss of generality, that in formula \eqref{ytdey6es6y} the partition $\pi=\{A_1,\dots,A_k\}$ is given by
\begin{equation}\label{ufy7fr} A_1=\{1,\dots,i_1\},\ A_2=\{i_1+1,\dots,i_1+i_2\},\ A_3=\{i_1+i_2+1,\dots,i_1+i_2+i_3\},\dots\end{equation}
Note that, since  $m^{(n)}$ is a Radon measure on $X^n$,  each measure $m_{i_1,\dots,i_k}$ is a Radon measure on $X^{(k)}_{\widehat0}$, i.e., for each $\Delta\in\mathcal B_c(X^{(k)}_{\widehat0})$, we have $m_{i_1,\dots,i_k}(\Delta)<\infty$. Here $\mathcal B_c(X^{(k)}_{\widehat0})$ denotes the collection of all sets $\Delta\in \mathcal B(X^{(k)}_{\widehat0})$ which have a compact closure in $X^k$, and  $ \mathcal B(X^{(k)}_{\widehat0})$ is the trace $\sigma$-algebra of $\mathcal B(X^k)$ on $X^{(k)}_{\widehat0}$.
Thus,  a given sequence of symmetric Radon measures $m^{(n)}$ on $X^n$, $n\in\mathbb N$, uniquely identifies a sequence
of Radon measures $m_{i_1,\dots,i_k}$ on $X^{(k)}_{\widehat0}$, where $i_1,\dots,i_k\in\mathbb N$, $k\in\mathbb N$.
Note that this sequence is symmetric in the entries $i_1,\dots,i_k$, i.e., for any permutation $\sigma\in\mathfrak S_k$,
$$dm_{i_1,\dots,i_k}(x_1,\dots,x_k)=dm_{i_{\sigma(1)},\dots,i_{\sigma(k)}}(x_{\sigma(1)},\dots,x_{\sigma(k)}).$$
As easily seen the converse implication is also true, i.e., any sequence of Radon measures $m_{i_1,\dots,i_k}$ on $X^{(k)}_{\widehat0}$, with $i_1,\dots,i_k\in\mathbb N$ and $k\in\mathbb N$ which is symmetric in the entries $i_1,\dots,i_k$ uniquely identifies a sequence of symmetric Radon measures $m^{(n)}$ on $X^n$, $n\in\mathbb N$.

Let now $\mu$ be a random discrete measure on $X$ which has finite moments, and let $(M^{(n)})_{n=0}^\infty$ be its moment sequence. Clearly, each $M^{(n)}$ is a symmetric measure on $X^n$.
Below we will deal with the measures $M_{i_1,\dots,i_k}$ derived from  the moment sequence $(M^{(n)})_{n=0}^\infty$.

For each $n\in\mathbb N$, we denote by $C_0(X^n)$  the space of all continuous functions on $X^n$ with compact support equipped with the natural topology  of uniform convergence on compact sets from $X^n$. Clearly, for each $f^{(n)}\in C_0(X^n)$, the function
$$\mathbb M(X)\ni\eta\mapsto\langle \eta^{\otimes n}, f^{(n)}\rangle:=\int_{X^n} f^{(n)}(x_1,\dots,x_n)\,d\eta(x_1)\dotsm d\eta(x_n)$$
is measurable. By the dominated convergence theorem it also holds that
\begin{equation}\label{bvufutf}
\int_{X^n} f^{(n)}(x_1,\dots,x_n)\, dM^{(n)}(x_1,\dots,x_n)=\int_{\mathbb M(X)}\langle \eta^{\otimes n}, f^{(n)}\rangle\,d\mu(\eta).\end{equation}

Consider the locally compact Polish space $Y=X\times \R_+$ (see Introduction), and consider the configuration space $\Gamma(Y)$. Recall that
$$\Gamma(Y):=\{\gamma\subset Y\mid |\gamma\cap\Lambda|<\infty\text{ for each compact }\Lambda\subset Y\}.$$
Here, $|\gamma\cap\Lambda|$ denotes the number of points in the set $\gamma\cap\Lambda$. One usually identifies a configuration $\gamma=\{y_i\}\in\Gamma(Y)$ with a Radon measure $\gamma=\sum_{i}\delta_{y_i}$. Thus, we get the inclusion $\Gamma(Y)\subset\mathbb M(Y)$ and we denote by $\mathcal B(\Gamma(Y))$ the  trace $\sigma$-algebra of $\mathcal B(\mathbb M(Y))$ on $\Gamma(Y)$.

 Denote by $\Gamma_p(Y)$  the set of so-called {\it pinpointing configurations in  $Y$}. By definition, $\Gamma_p(Y)$ consists of all configurations $\gamma\in\Gamma(Y)$ such that if $(x_1,s_1),(x_2,s_2)\in\gamma$ and $(x_1,s_1)\ne(x_2,s_2)$, then $x_1\ne x_2$. Thus, a configuration $\gamma\in \Gamma_p(Y)$ cannot contain two points $(x,s_1)$ and $(x,s_2)$ with $s_1\ne s_2$.
For each $\gamma\in\Gamma_p(Y)$ and $\Lambda\in\mathcal B_c(X)$, we define a {\it local mass} by
\begin{equation}\label{ddddydyyd}
\mathfrak{M}_\Lambda(\gamma):= \int_{Y}\chi_\Lambda(x)s\,d\gamma(x,s)=\sum_{(x,s)\in\gamma}\chi_\Lambda(x)s\in[0,\infty].\end{equation}
Here $\chi_\Lambda$ denotes the indicator function of the set $\Lambda$.
The set of {\it pinpointing configurations with finite local mass\/} is then defined by
\begin{equation}\label{fdstrset6s}
\Gamma_{pf}(Y):=\big\{\gamma\in \Gamma_p(Y)\mid \mathfrak{M}_\Lambda(\gamma)<\infty\text{ for each compact }\Lambda\subset X\big\}.\end{equation}
As easily seen, $\Gamma_{pf}(Y)\in\mathcal B(\Gamma(Y))$ and we denote by $\mathcal B(\Gamma_{pf}(Y))$ the trace $\sigma$-algebra of  $\mathcal B(\Gamma(Y))$ on $\Gamma_{pf}(Y)$.

We  construct a bijective mapping
$\mathcal E:\mathbb K(X)\to \Gamma_{pf}(Y)$ by setting, for each $\eta=\sum_i s_i\delta_{x_i}\in \mathbb K(X)$,
$\mathcal E\eta:=\{(x_i,s_i)\}$.
By \cite[Theorem~6.2]{HKPR}, we have
$$\mathcal B(\Gamma_{pf}(Y))=\left\{\mathcal EA\mid A\in\mathcal B(\mathbb K(X)\right\}.$$
Hence, both $\mathcal E$ and its inverse $\mathcal E^{-1}$ are measurable mappings.

We denote by $\nu:=\mathcal E(\mu)$ the pushforward of the measure $\mu$ under the mapping $\mathcal E$. Thus $\nu$ is a probability measure on $\Gamma_{pf}(Y)$, in particular, it is  a point process in $Y$.

Let $\Gamma_0(Y)$ denote the space of all finite configurations in $Y$:
$$\Gamma_0(Y):=\{\gamma\subset Y\mid |\gamma|<\infty\}.$$ Note that $\Gamma_0(Y)=\bigcup_{n=0}^\infty\Gamma^{(n)}(Y)$, where $\Gamma^{(n)}(Y)$ is the space of all $n$-point configurations (subsets) in $Y$. Clearly, $\Gamma_0(Y)\subset\Gamma(Y)$, and we denote by $\mathcal B(\Gamma_0(Y))$ the trace $\sigma$-algebra of $\mathcal B(\Gamma(Y))$ on $\Gamma_0(Y)$. The $\sigma$-algebra $\mathcal B(\Gamma_0(Y))$ admits  the following description: for each $n\in\mathbb N$, $\Gamma^{(n)}(Y)\in\mathcal B(\Gamma_0(Y))$ and the restriction of $\mathcal B(\Gamma_0(Y))$ to $\Gamma^{(n)}(Y)$ coincides (up to a natural isomorphism)
with the collection of all symmetric (i.e., invariant under the action of $\sigma\in\mathfrak S_n$) Borel-measurable subsets of $Y^{(n)}_{\widehat 0}$.
 The {\it correlation measure of the point process $\nu$} is
defined as the (unique) measure
$\rho$ on $(\Gamma_0(Y),\mathcal B(\Gamma_0(Y)))$ which satisfies
\begin{equation}\label{yuf7ur7ddddytr}\int_{\Gamma(Y)}\sum_{\lambda\Subset\gamma} G(\lambda)\,d\nu(\gamma)=\int_{\Gamma_0(Y)}G(\lambda)\,d\rho(\lambda) \end{equation}
for each measurable function $G:\Gamma_0(Y)\to[0,\infty]$. In formula \eqref{yuf7ur7ddddytr}, the summation $\sum_{\lambda\Subset\gamma}$ is over all finite subsets $\lambda$ of $\gamma$.

 For each $n\in\mathbb N$, we denote by $\rho^{(n)}$ the restriction of the measure $\rho$ to $\Gamma^{(n)}(Y)$. By \eqref{yuf7ur7ddddytr}, the measure $\rho^{(n)}$ can be identified with the symmetric measure on $Y^{(n)}_{\widehat0}$ which satisfies
\begin{align}
&\int_{\Gamma_{pf}(Y)}\sum_{\{(x_1,s_1),\dots,(x_n,s_n)\}\subset\gamma} f^{(n)}(x_1,s_1,\dots,x_n,s_n)\,d\nu(\gamma)\notag\\
&\quad=\int_{Y^{(n)}_{\widehat 0}}f^{(n)}(x_1,s_1,\dots,x_n,s_n)\,d\rho^{(n)}(x_1,s_1,\dots,x_n,s_n)\label{vuduy}
\end{align}
for each symmetric measurable function $f^{(n)}:Y^{(n)}_{\widehat 0}\to[0,\infty]$.
Since  $\nu(\Gamma_{p}(Y))=1$, the  measure $\rho^{(n)}$ is concentrated  on the smaller set
\begin{equation}\label{huyfr8ultrfi} \set_n :=\big\{(x_1,s_1,\dots,x_n,s_n)\in Y^n\mid (x_1,\dots,x_n)\in X^{(n)}_{\widehat0}\big\}.
\end{equation}

The following theorem gives a three-step way of recovering the correlation measure $\rho$ of the point process $\nu=\mathcal E(\mu)$ directly from the moment sequence $(M^{(n)})_{n=0}^\infty$.

\begin{theorem}\label{ufut7fruvfguqfd}
Let  $\mu$ be a random discrete measure on $X$ which has finite moments. Let $(M^{(n)})_{n=0}^\infty$ be the moment sequence of $\mu$,  and assume that condition  {\rm(C1)} is satisfied.

\begin{itemize}

\item[(i)] For each  $n\in\mathbb N$ and  $\Delta\in\mathcal B_c(X^{(n)}_{\widehat0})$,
there exists a unique finite measure $\xi_\Delta^{(n)}$ on $( \R_+)^n$ which solves the moment problem
\begin{equation}\label{fstrst6e6}
\int_{( \R_+)^n} s_1^{i_1}\dotsm s_n^{i_n}\, d\xi^{(n)}_\Delta(s_1,\dots,s_n)=\frac1{n!}\,
M_{i_1+1,\dots,i_n+1}(\Delta),\quad (i_1,\dots,i_n)\in\mathbb Z_+^n. \end{equation}

\item[(ii)] For each  $n\in\mathbb N$,
there exists a unique  measure $\xi^{(n)}$ on $\set_n$ which satisfies
\begin{equation}\label{ufru8} \xi_\Delta^{(n)}(A)=\int_{\set_n}\chi_\Delta(x_1,\dots,x_n)\chi_{A}(s_1,\dots,s_n)\,d\xi^{(n)}
(x_1,s_1,\dots,x_n,s_n)\end{equation}
for all   $\Delta\in\mathcal B_c(X^{(n)}_{\widehat0})$ and $A\in\mathcal B(( \R_+)^n)$.

\item[(iii)] For each  $n\in\mathbb N$, let $\rho^{(n)}$ be the measure on $\set_n$ given by
\begin{equation}\label{uyr76or8o}d\rho^{(n)}(x_1,s_1,\dots,x_n,s_n):=(s_1\dotsm s_n)^{-1} \,d\xi^{(n)}(x_1,s_1,\dots,x_n,s_n).\end{equation}
Then $\rho^{(n)}$ is the restriction of the correlation measure $\rho$ of the point process $\nu=\mathcal E(\mu)$ to $\Gamma^{(n)}(X)$.

\end{itemize}

\end{theorem}

\begin{remark} Note that, by the definition of a correlation measure, one always has $\rho(\varnothing)=1$.
Note also that $\rho^{(n)}$ is related to $(M^{(n)})_{n=0}^\infty$ via a moment problem, because, as shown in the proof, the following relation holds, for any  measurable function $g^{(n)}: X^{(n)}_{\widehat 0}\to[0,\infty]$,
\begin{align}
&\int_{\set_n}g^{(n)}(x_1,\dots,x_n)s_1^{i_1}\dots, s_n^{i_n}\,d\rho^{(n)}(x_1,s_1,\dots,x_n,s_n).\notag\\
&\quad=\frac1{n!}\int_{X^{(n)}_{\widehat 0}}g^{(n)}(x_1,\dots,x_n)\,dM_{i_1,\dots,i_n}(x_1,\dots,x_n).\label{qg8yodw}
\end{align}

\end{remark}

\begin{proof}[Proof of Theorem \ref{ufut7fruvfguqfd}] We start the proof with  derivation of the following bound.

\begin{lemma}\label{lhgiuty9} Assume that, for each $n\in\mathbb N$, $m^{(n)}$ is a symmetric measure on $X^n$. Assume that, for each $\Lambda\in\mathcal B_c(X)$, there exists a constant $C_\Lambda>0$ such that
$m^{(n)}(\Lambda^n)\le C_\Lambda^n\,n!$ for all $n\in\mathbb N$. Then, for any $i_1,\dots,i_n\in\mathbb N$, $n\in\mathbb N$, and $\Lambda\in\mathcal B_c(X)$,
$$ \frac1{n!}\, m_{i_1,\dots,i_n}(\Lambda_{\widehat 0}^{(n)})\le i_1!\dotsm i_n!\, C_\Lambda^{i_1+\dots+i_n}.$$
\end{lemma}

\begin{proof} Fix any $i_1,\dots,i_n\in\mathbb N$ and $\Lambda\in\mathcal B_c(X)$. Abbreviate $I := i_1+\dots+i_n$. Let $\pi=\{A_1,\dots,A_n\}\in\Pi(I)$ be  as in \eqref{ufy7fr}. By the construction of the measure $m_{i_1,\dots,i_n}$, we get
\begin{align}
&m_{i_1,\dots,i_n}(\Lambda_{\widehat 0}^{(n)})\notag\\
&\quad=\int_{X_\pi^{(I)}}\chi_{\Lambda^n}(x_1,x_{i_1+1},\dots,x_{i_1+\dots+i_{n-1}+1})\,
dm^{(I)}(x_1,\dots, x_{I})\notag\\
&\quad=\int_{X^{I}}
\chi_{\Lambda^{I}\cap X_\pi^{(I)}}(x_1,\dots,x_{I})\,
dm^{(I)}(x_1,\dots, x_{I})\notag\\
&\quad=\int_{X^{I}} \chi_{\Lambda_\pi^{(I)}}\,
dm^{(I)}\notag\\
&\quad=\int_{X^{I}} \operatorname{Sym}_{I}\chi_{\Lambda_\pi^{(I)}}\,
dm^{(I)}.
\label{hude7k}\end{align}
Here, for a function $f^{(k)}:X^k\to\mathbb R$, $\operatorname{Sym}_k f^{(k)}$ denotes its symmetrization:
$$ (\operatorname{Sym}_k f^{(k)})(x_1,\dots,x_k):=\sum_{\sigma\in\mathfrak S_k}\frac1{k!}\, f(x_{\sigma(1)},\dots,x_{\sigma(k)}).$$

Let $\psi\in\Pi(I)$ be a partition having exactly $n$ blocks:
$$\psi=\{B_1,\dots,B_n\}.$$
 Set $j_l:=|B_l|$, $l=1,\dots,n$.
Denote by $\Psi_{i_1,\dots,i_n}$ the set of all such partitions $\psi$ which satisfy
$$(i_1,\dots,i_n)=(j_{\sigma(1)},\dots, j_{\sigma(n)})$$
for some permutation $\sigma\in\mathfrak S_n$. An easy combinatoric argument shows that the number $N_{i_1,\dots,i_n}$ of all partitions in $\Psi_{i_1,\dots,i_n}$ is equal to
\begin{equation}\label{eqfactorcombipartition}
N_{i_1,\dots,i_n}=\frac{I!}{i_1!\,\dotsm\, i_n!\,r_1!\,r_2!\,r_3!\dotsm}\,. \end{equation}
Here for $l=1,2,3,\dots$, $r_l$ denotes the number of coordinates in the vector $(i_1,i_2,\dots,i_n)$ which are equal $l$.
In particular,
$$r_1+r_2+r_3+\dotsm=n,$$
which implies
$$r_1!\,r_2!\,r_3!\dotsm\le n!\,.$$
Therefore,
\begin{equation}\label{huyrf}N_{i_1,\dots,i_n}\ge \frac{I!}{i_1!\,\dotsm\, i_n!\,n!}\,. \end{equation}
For each $\psi\in\Psi_{i_1,\dots,i_n}$,
$$ \operatorname{Sym}_{I}\chi_{\Lambda_\psi^{(I)}}
=\operatorname{Sym}_{I}\chi_{\Lambda_\pi^{(I)}}.$$
Hence, by \eqref{hude7k} and \eqref{huyrf},
\begin{align}
&\frac1{n!}\,m_{i_1,\dots,i_n}(\Lambda_{\widehat 0}^{(n)})\notag\\
&\quad=\frac1{n!\, N_{i_1,\dots,i_n}}\sum_{\psi\in \Psi_{i_1,\dots,i_n}}
\int_{X^{I}} \chi_{\Lambda_\psi^{(I)}}\,
dm^{(I)}\notag\\
&\quad\le \frac{i_1!\,\dotsm i_n!}{I!}\,\int_{X^{I}}
\sum_{\psi\in \Psi_{i_1,\dots,i_n}}\chi_{\Lambda_\psi^{(I)}}\,
dm^{(I)}\notag\\
&\quad\le \frac{i_1!\,\dotsm i_n!}{I!}\,m^{(I)}(\Lambda^{I})\notag\\
&\quad\le i_1!\,\dotsm i_n!\, C_\Lambda^{I}.\notag
\end{align}
\end{proof}

To prove statements (i)--(iii) of the theorem, let us first  carry out some considerations. Note that, for each $n\in\mathbb N$ and each measurable function $f^{(n)}:X^n\to[0,\infty]$, the functional
$$\mathbb K(X)\ni\eta\mapsto\langle \eta^{\otimes n},f^{(n)}\rangle\in[0,\infty]$$ is measurable and
\begin{equation}\label{ufry7urfd}
\int_{\mathbb K(X)}\langle\eta^{\otimes n},f^{(n)}\rangle\,d\mu(\eta)=\int_{X^n}f^{(n)}\,dM^{(n)}.
\end{equation}

As easily seen, equality \eqref{vuduy} can be extended to the class of all measurable  (not necessarily symmetric) functions $f^{(n)}:\set_n
\to[0,\infty]$ as follows:
\begin{align}
&\int_{\Gamma_{pf}(Y)}\frac1{n!}\sum_{\substack{(x_1,s_1),\dots,(x_n,s_n)\in\gamma\\
\text{$x_1,\dots,x_n$ different}}} f^{(n)}(x_1,s_1,\dots,x_n,s_n)\,d\nu(\gamma)\notag\\
&\quad=\int_{\set_n}f^{(n)}(x_1,s_1,\dots,x_n,s_n)\,d\rho^{(n)}(x_1,s_1,\dots,x_n,s_n).\label{futkdr7keddxrs}
\end{align}
If we extend the function $f^{(n)}$ by zero to the whole space $Y^n$, we can rewrite \eqref{futkdr7keddxrs} in the equivalent form:
\begin{align}
&\int_{\Gamma_{pf}(Y)}\frac1{n!}\sum_{(x_1,s_1),\dots,(x_n,s_n)\in\gamma} f^{(n)}(x_1,s_1,\dots,x_n,s_n)\,d\nu(\gamma)\notag\\
&\quad=\int_{\set_n}f^{(n)}(x_1,s_1,\dots,x_n,s_n)\,d\rho^{(n)}(x_1,s_1,\dots,x_n,s_n).\label{gdQF}
\end{align}
In particular, for any measurable function $g^{(n)}:X^n\to[0,\infty]$ which vanishes outside $X^{(n)}_{\widehat 0}$ and any $i_1,\dots,i_n\in\mathbb N$, we get
\begin{align}
&\int_{\Gamma_{pf}(Y)}\frac1{n!}\sum_{(x_1,s_1),\dots,(x_n,s_n)\in\gamma} g^{(n)}(x_1,\dots,x_n)s_1^{i_1}\dotsm s_n^{i_n}\,d\nu(\gamma)\notag\\
&\quad=\int_{\set_n}g^{(n)}(x_1,\dots,x_n)s_1^{i_1}\dots, s_n^{i_n}\,d\rho^{(n)}(x_1,s_1,\dots,x_n,s_n).\label{dfqytd}
\end{align}

For each function $f:X_\pi^{(i_1+ \ldots + i_n)} \rightarrow \mathbb{R}$ one can define a function $X^{(n)}_{\widehat 0} \rightarrow \mathbb{R}$ via $B_\pi:X_\pi^{(i_1+\dots+i_n)}\to X^{(n)}_{\widehat 0}$. Now we will describe the opposite procedure.
For simplicity of notation, we will write below
$$\mathcal{I}_n(x_1,\dots,x_n):=\chi_{X^{(n)}_{\widehat1}}(x_1,\dots,x_n),\quad (x_1,\dots,x_n)\in X^n.$$
Thus, $\mathcal{I}_n(x_1,\dots,x_n)$ is
equal to 1 if $x_1=x_2=\dots=x_n$, and is equal to zero otherwise.
For $i_1,\dots,i_n\in\mathbb N$, we define a function $\mathcal I_{i_1,\dots,i_n}:X^{i_1+\dots+i_n}\to\{0,1\}$ by setting
\begin{align*}&\mathcal I_{i_1,\dots,i_n}(x_1,\dots,x_{i_1+\dots+i_n})\\
&\quad:=\mathcal{I}_{i_1}(x_1,\dots,x_{i_1})\mathcal{I}_{i_2}(x_{i_1+1},\dots,x_{i_1+i_2})\dotsm
\mathcal{I}_{i_n}(x_{i_1+\dots+i_{n-1}+1},\dots,x_{i_1+\dots+i_n}).\end{align*}
For a measurable function $g^{(n)}:X^n\to[0,\infty)$ which vanishes outside $X^{(n)}_{\widehat 0}$, we define a measurable function
$\mathcal R_{i_1,\dots,i_n}g^{(n)}:X^{i_1+\dots+i_n}\to [0,\infty]$ by
\begin{align}&(\mathcal R_{i_1,\dots,i_n}g^{(n)})(x_1,\dots,x_{i_1+\dots+i_n})\notag\\
&\quad:= g^{(n)}(x_1,x_{i_1+1},x_{i_1+i_2+1},\dots,x_{i_1+\dots+i_{n-1}+1})\mathcal I_{i_1,\dots,i_n}(x_1,\dots,x_{i_1+\dots+i_n}).\label{jhacvc}\end{align}
Note that the function $\mathcal R_{i_1,\dots,i_n}g^{(n)}$ vanishes outside the set $X^{(i_1+\dots+i_n)}_\pi$, where  $\pi=\{A_1,\dots,A_n\}$  with the sets $A_1,\dots,A_n$ being as in \eqref{ufy7fr}.
For each $\eta\in\mathbb K(X)$,
\begin{align}
&\langle \eta^{\otimes(i_1+\dots+i_n)},\mathcal R_{i_1,\dots,i_n}g^{(n)}\rangle\notag\\
&\quad=\sum_{(x_1,s_1),\dots,(x_I,s_I)\in\mathcal E(\eta)}(\mathcal R_{i_1,\dots,i_n}g^{(n)})(x_1,\dots,x_{I})s_1\dotsm s_{I}\notag\\
&\quad=n! \sum_{\{(x_1,s_1),\dots,(x_n,s_n)\}\subset\mathcal E(\eta)}g^{(n)}(x_1,\dots,x_n)s_1^{i_1}\dotsm s_n^{i_n}.\label{vhufut7f}
\end{align}
Here we write $I=i_1+\dots+i_n$ to save the space.
By \eqref{dfqytd}, \eqref{vhufut7f}, and the definition of the measure $\nu$, we get
\begin{align}
&\frac1{n!}\int_{\mathbb K(X)}\langle \eta^{\otimes(i_1+\dots+i_n)},\mathcal R_{i_1,\dots,i_n}g^{(n)}\rangle\,d\mu(\eta)\notag\\
&\quad=\int_{\set_n}g^{(n)}(x_1,\dots,x_n)s_1^{i_1}\dots, s_n^{i_n}\,d\rho^{(n)}(x_1,s_1,\dots,x_n,s_n).\notag
\end{align}
Hence, by \eqref{ufry7urfd},
\begin{align}
&\int_{\set_n}g^{(n)}(x_1,\dots,x_n)s_1^{i_1}\dots, s_n^{i_n}\,d\rho^{(n)}(x_1,s_1,\dots,x_n,s_n).\notag\\
&\quad=\frac1{n!}
\int_{X^{i_1+\dots+i_n}}\mathcal R_{i_1,\dots,i_n}g^{(n)}\, dM^{(i_1+\dots+i_n)}
\notag\\
&\quad=\frac1{n!}
\int_{X^{(i_1+\dots+i_n)}_\pi}\mathcal R_{i_1,\dots,i_n}g^{(n)}\, dM^{(i_1+\dots+i_n)},
\notag
\end{align}
where the partition $\pi$ is as above. From here we conclude that equality \eqref{qg8yodw} holds.
We  define a symmetric measure $\xi^{(n)}$ on $\set_n$ by setting
\begin{equation}\label{gvhfcdtudf}
d\xi^{(n)}(x_1,s_1,\dots,x_n,s_n):=s_1\dotsm s_n \,d\rho^{(n)}(x_1,s_1,\dots,x_n,s_n).\end{equation}
Then, equality \eqref{qg8yodw} can be rewritten as follows:
\begin{align}
&\int_{\set_n}g^{(n)}(x_1,\dots,x_n)s_1^{i_1}\dots, s_n^{i_n}\,d\xi^{(n)}(x_1,s_1,\dots,x_n,s_n).\notag\\
&\quad=\frac1{n!}\int_{X^{(n)}_{\widehat 0}}g^{(n)}(x_1,\dots,x_n)\,dM_{i_1+1,\dots,i_n+1}(x_1,\dots,x_n),
\quad (i_1,\dots,i_n)\in\mathbb Z_+^n.
\label{kigfti8ylt}
\end{align}

For any $\Delta\in\mathcal B_c(X^{(n)}_{\widehat0})$, let $\xi_\Delta^{(n)}$ be the finite measure on $( \R_+)^n$ which satisfies \eqref{ufru8}.
Denote
\begin{equation}\label{ude7re7u}\xi^\Delta_{\mathbf i}=\xi^\Delta_{i_1,\dots,i_n}:=\frac1{n!}\,M_{i_1+1,\dots,i_n+1}(\Delta),\quad \mathbf i=(i_1,\dots,i_n)\in\mathbb Z_+^n.\end{equation}
Then, by \eqref{kigfti8ylt} and \eqref{ude7re7u},
\begin{equation}\label{bhutfdy7uf}\xi^\Delta_{\mathbf i}=\int_{( \R_+)^n} s_1^{i_1}\dotsm s_n^{i_n}\,d\xi_\Delta^{(n)}(s_1,\dots,s_n) ,\quad \mathbf i=(i_1,\dots,i_n)\in\mathbb Z_+^n.\end{equation}
Thus, $(\xi^\Delta_{\mathbf i})_{\mathbf i\in\mathbb Z_+^n}$ is the moment sequence of the  finite  measure $\xi_\Delta^{(n)}$.

Choose any $\Lambda\in\mathcal B_c(X)$ such that $\Delta\subset\Lambda^{(n)}_{\widehat0}$. By formulas \eqref{hydyrk}, \eqref{ude7re7u} and Lemma~\ref{lhgiuty9},
\begin{align}
\xi_{i_1,\dots,i_n}^\Delta&\le \frac1{n!}\, M_{i_1+1,\dots,i_n+1}(\Lambda_{\widehat 0}^{(n)})\notag\\
&\le (i_1+1)!\,\dotsm (i_n+1)!\, C_{\Lambda}^{i_1+\dots+i_n+n}\notag\\
&\le (i_1+\dots+i_n+n)!\, C_\Lambda ^{i_1+\dots+i_n+n},\quad (i_1,\dots,i_n)\in\mathbb Z_+^n.\label{gyit9i78t5798pt5r86t}
\end{align}

We are now ready to finish the proof of the theorem. Since $(\xi^\Delta_{\mathbf i})_{\mathbf i\in\mathbb Z_+^n}$ is the moment sequence of the   finite  measure $\xi_\Delta^{(n)}$ on $( \R_+)^n$, and since this moment sequence satisfies  estimate
\eqref{gyit9i78t5798pt5r86t}, we conclude from e.g.\ \cite[Chapter 5, Subsec.~2.1, Examples 2.1, 2.2]{BK} that the moment sequence $(\xi^\Delta_{\mathbf i})_{\mathbf i\in\mathbb Z_+^n}$ uniquely identifies the measure $\xi_\Delta^{(n)}$. Hence, statement (i)  holds. Next, equality \eqref{ufru8} evidently holds by (\ref{kigfti8ylt}). Note also that the values of the measure $\xi^{(n)}$ on the sets of the form
$$\big\{(x_1,s_1,\dots,x_n,s_n)\in \set_n\mid (x_1,\dots,x_n)\in\Delta,\ (s_1,\dots,s_n)\in A\big\}$$
where $\Delta\in\mathcal B_c(X^{(n)}_{\widehat0})$ and $A\in\mathcal B(( \R_+)^n)$,  completely identify the measure $\xi^{(n)}$ on $\set^n$. Thus, statement (ii) holds. Finally, statement (iii) trivially follows from \eqref{gvhfcdtudf}.

\end{proof}

\section{A characterization  of  random discrete measure in terms of moments}\label{hftuttfd7u}


In this section, we assume that $\mu$ is a random measure on $X$ which has finite moments. Let $(M^{(n)})_{n=0}^\infty$ be its moment sequence.  We assume additionally to  condition (C1) that  condition (C2) is satisfied.

\begin{remark}\label{rmc2}
Assumption (C2) is usually satisfied by a  measure $\mu$  being concentrated on the cone $\mathbb K(X)$.
In the latter case, by the proof of Theorem~\ref{ufut7fruvfguqfd}, we have
\begin{align*}
M^{(n)}(\Lambda^{(n)}_{\widehat0})&=M_{1, \dots , 1}(\Lambda^{(n)}_{\widehat0})=n!\,\xi^{(n)}(\set_n\cap (\Lambda\times \R_+)^n)\\
&= n!\int_{\set_n\cap (\Lambda\times \R_+)^n}
s_1\dotsm s_n\,d\rho^{(n)}(x_1,s_1,\dots,x_n,s_n),
\end{align*}
so that estimate \eqref{lhgigikbgi} becomes
$$\int_{\set_n\cap (\Lambda\times \R_+)^n}
s_1\dotsm s_n\,d\rho^{(n)}(x_1,s_1,\dots,x_n,s_n)\le (C_\Lambda')^n.$$
For example, in the case of the gamma measure (see Introduction), we have
$$ \int_{\set_n\cap (\Lambda\times \R_+)^n}
s_1\dotsm s_n\,d\rho^{(n)}(x_1,s_1,\dots,x_n,s_n)=\frac1{n!}\left(\int_{\Lambda}dx\right)^n,$$
so condition (C2) is trivially satisfied.

Note also that one should not  expect that the constant $C_\Lambda$ in estimate (C1) becomes small as set $\Lambda$ shrinks to an empty set. This, for example, is not even true in the case of the gamma measure.
Indeed, $$M^{(n)}(\Lambda) =
\prod_{k=0}^{n-1}\left(\int_\Lambda dx+k\right).
$$
($M^{(n)}(\Lambda)$ is the $n$-th moment of the gamma distribution with parameter $\int_\Lambda dx$.)
 For each $n$, this decays at most like $\int_\Lambda dx$ and hence $C_\Lambda$ cannot decrease to zero. 
 \end{remark}

Recall that before Theorem~\ref{jkgfrt7urd} we fixed  a sequence $(\Lambda_l)_{l=1}^\infty$ of compact subsets of $X$ such that $\Lambda_1\subset\Lambda_2\subset\Lambda_3\subset\dotsm$ and $\bigcup_{l=1}^\infty\Lambda_l=X$.

\begin{proof}[Proof of Theorem~\ref{jkgfrt7urd}]
Assume that $\mu(\mathbb K(X))=1$ and let us show that conditions (i) and (ii) are satisfied.
Let $\Delta\in\mathcal B_c(X^{(n)}_{\widehat0})$.
It follows from the proof of  Theorem~\ref{ufut7fruvfguqfd} (see in particular formula \eqref{bhutfdy7uf}) that  the sequence $(\xi^\Delta_{\mathbf i})_{\mathbf i\in\mathbb Z_+^n}$ is the moment sequence of the  finite  measure $\xi_\Delta^{(n)}$. Hence, condition (i) is indeed satisfied (see e.g.\ \cite[Chapter 5, Subsec.~2.1]{BK}).

Next, let $\Delta\in\mathcal B(X_{\widehat 0}^{(n)})$
be of the form $\Delta=(\Lambda_l)^{(n)}_{\widehat 0}$. Clearly,
 $(r_i^\Delta)_{i=0}^\infty$ is the moment sequence of the first coordinate projection of the measure $\xi_\Delta^{(n)}$, which we denote by $P_1\xi_{\Delta}^{(n)}$.
The  measure $P_1\xi_{\Delta}^{(n)}$ is concentrated on $[0,\infty)$, hence \eqref{bkjgtfi8ytf8gi} follows (see e.g.\ \cite[Chapter~2, Subsec.~6.5]{A}).  By (C1), \eqref{ufru8},
\eqref{kigfti8ylt},  Lemma~\ref{lhgiuty9} and as $\Delta=(\Lambda_l)^{(n)}_{\widehat 0}$
\begin{align}
 r_i^\Delta&=\int_{( \R_+)^n} s_1^i\,d\xi^{(n)}_\Delta(s_1,\dots,s_n)\notag\\
&=\int_{\set_n} \chi_\Delta(x_1,\dots,x_n)s_1^{i}\,d\xi^{(n)}(x_1,s_1,\dots,x_n,s_n)\notag\\
&=\frac1{n!}\int_{X^{(n)}_{\widehat0}}\chi_\Delta(x_1,\dots,x_n) \,dM_{i+1,1,1,\dots,1}(x_1,\dots,x_n)\notag\\
&= \frac1{n!}\,M_{i+1,1,1,\dots,1}((\Lambda_l)^{(n)}_{\widehat0})\notag\\
&\le (i+1)!\, C_\Lambda^{n+i},\quad i\in\mathbb Z_+.
\end{align}
Hence, by the Carleman criterion (see e.g.\ \cite{A}), the measure $P_1\xi_{\Delta}^{(n)}$ is the unique measure on $\R$ which has moments $(r_i^\Delta)_{i=0}^\infty$. Therefore, by \cite[formula (4) in Chapter I, Sect.1; Chapter II, Subsec.~4.1; Theorem~2.5.3]{A}, formula
\eqref{ug86r8o7} follows from the fact that the measure $P_1\xi_{\Delta}^{(n)}$ has no atom at point 0. Thus, condition (ii) is satisfied.

\begin{remark} Note that, in this part of the proof,  we have not used condition (C2).
\end{remark}

Let us now prove the converse statement. So, we assume that $(M^{(n)})_{n=0}^\infty$ is a sequence of symmetric Radon measures fulfilling (C1), (C2),  (i), (ii) and $M^{(0)}=1$. That $(M^{(n)})_{n=0}^\infty$ is the moment sequence of a probability measure $\mu$ on $(\mathbb M(X),\mathcal B(\mathbb M(X)))$ will only be  used in Lemma~\ref{utft7u7r}. We will show the existence of a measure $\mu'$ with $\mu'(\mathbb K(X))=1$ which has as its moments $(M^{(n)})_{n=0}^\infty$. Finally, we will argue that, if $(M^{(n)})_{n=0}^\infty$ is the moment sequence of a probability measure $\mu$, then by the uniqueness of solution of the moment problem $\mu=\mu'$ and hence $\mu(\mathbb K(X))=1$.

Fix any $n\in\mathbb N$ and $\Delta\in\mathcal B_c(X^{(n)}_{\widehat 0})$.
 Choose  $\Lambda\in\mathcal B_c(X)$ such that $\Delta\subset\Lambda^{(n)}_{\widehat0}$. By  (C1), \eqref{ude7re7uvt}, and Lemma~\ref{lhgiuty9},
\begin{align}
\xi_{i_1,\dots,i_n}^\Delta&=\frac1{n!}\, M_{i_1+1,\dots,i_n+1}(\Delta)\notag\\
&\le \frac1{n!}\, M_{i_1+1,\dots,i_n+1}(\Lambda_{\widehat 0}^{(n)})\notag\\
&\le (i_1+1)!\,\dotsm (i_n+1)!\, C_{\Lambda}^{i_1+\dots+i_n+n}\notag\\
&\le (i_1+\dots+i_n+n)!\, C_\Lambda ^{i_1+\dots+i_n+n},\quad (i_1,\dots,i_n)\in\mathbb Z_+^n.\label{it68r}
\end{align}
Furthermore, by condition (i), the sequence $(\xi^\Delta_{\mathbf i})_{\mathbf i\in\mathbb Z_+^n}$ is positive definite. Hence, using e.g.\ \cite[Chapter 5, Subsec.~2.1, Examples 2.1, 2.2]{BK},  we conclude that there exists a unique measure $\xi_\Delta^{(n)}$ on $\R^n$ such that $(\xi^\Delta_{\mathbf i})_{\mathbf i\in\mathbb Z_+^n}$ is its moment sequence, i.e.,
\begin{equation}\label{tyde6ue6iut}
\xi^\Delta_{\mathbf i}=\int_{\R^n} s_1^{i_1}\dotsm s_n^{i_n}\,d\xi_\Delta^{(n)}(s_1,\dots,s_n) ,\quad \mathbf i=(i_1,\dots,i_n)\in\mathbb Z_+^n.\end{equation}

\begin{lemma}\label{huf}
Let $n\in\mathbb N$. Let $\{\Delta_k\}_{k=1}^\infty$ be a sequence of disjoint sets from $\mathcal B_c(X^{(n)}_{\widehat0})$. Denote $\Delta:=\bigcup_{k=1}^\infty\Delta_k$ and  assume that $\Delta\in\mathcal B_c(X^{(n)}_{\widehat0})$. We then have
\begin{equation}\label{igti88}
\sum_{k=1}^\infty \xi^{(n)}_{\Delta_k}=\xi_{\Delta}^{(n)}.
\end{equation}

\end{lemma}

\begin{proof}
Fix  any $i_1,\dots,i_n\in\mathbb Z_+$. Since $M^{(i_1+\dots+ i_n)}$ is a measure, we easily get
\begin{align*}&\sum_{k=1}^\infty \int_{\R^n}s_1^{i_1}\dotsm s_n^{i_n}\,d\xi^{(n)}_{\Delta_k}(s_1,\dots,s_n)\\
&\quad =\sum_{k=1}^\infty \frac1{n!}\,M_{i_1+1,\dots,i_n+1}(\Delta_k)\\
&\quad= \frac1{n!}M_{i_1+1,\dots,i_n+1}(\Delta)
\\
&\quad=\int_{\R^n}s_1^{i_1}\dotsm s_n^{i_n}\,d\xi^{(n)}_{\Delta}(s_1,\dots,s_n).\end{align*}
Hence, the measures $\sum_{k=1}^\infty \xi^{(n)}_{\Delta_k}$ and $\xi_\Delta^{(n)}$
have the same moments.
The measure $\xi^{(n)}$ fulfils the Carleman bound,  hence it  is uniquely identified by its moments. So \eqref{igti88} holds. \end{proof}

\begin{lemma}\label{hfutf}
For any $n\in\mathbb N$ and $\Delta\in\mathcal B_c(X^{(n)}_{\widehat 0})$, the measure $\xi_\Delta^{(n)}$ is concentrated on $(\R_+)^n$.
\end{lemma}

\begin{proof}
Fix any $l\in\mathbb N$ and set $\Delta=(\Lambda_l)^{(n)}_{\widehat 0}$. By  \eqref{yufu8r86orftf} and \eqref{tyde6ue6iut},
$$ r_i^\Delta= \int_{\R^n}s_1^i \,d\xi_\Delta^{(n)}(s_1,\dots,s_n),\quad i\in\mathbb Z_+. $$
Thus, the numbers $(r_i^\Delta)_{i=0}^\infty$ form the moment sequence of the first coordinate projection of the measure $\xi_\Delta^{(n)}$, which we denote, as above, by $P_1\xi_\Delta^{(n)}$. As easily follows from \eqref{it68r} and the Carleman criterion, the measure $P_1\xi_\Delta^{(n)}$ is uniquely identified by its moment sequence. Then, by \eqref{bkjgtfi8ytf8gi}, the measure $P_1\xi_\Delta^{(n)}$ is concentrated on $[0,\infty)$, and by \eqref{ug86r8o7}, $(P_1\xi_\Delta^{(n)})(\{0\})=0$, see \cite{A}. Therefore, the measure $P_1\xi_\Delta^{(n)}$ is concentrated on $ \R_+$. Evidently, for any $(i_1,\dots,i_n)\in\mathbb Z_+^n$ and any $\sigma\in\mathfrak S_n$, we get
\begin{align*}&\int_{\R^n}s_{\sigma(1)}^{i_1}\dotsm s_{\sigma(n)}^{i_n}\,d\xi^{(n)}_\Delta(s_1,\dots,s_n)\\
&\quad=\int_{\R^n}
s_1^{i_{\sigma^{-1}(1)}}\dotsm s_n^{i_{\sigma^{-1}(n)}}\,d\xi^{(n)}_\Delta(s_1,\dots,s_n)\\
&\quad=\frac1{n!}\,M_{i_{\sigma^{-1}(1)+1}, \dots , i_{\sigma^{-1}(n)+1}}(\Delta)\\
&\quad=\frac1{n!}\,M_{i_1+1, \dots , i_n+1}(\Delta)\\
&\quad=\int_{\R^n}
s_1^{i_1}\dotsm s_n^{i_n}\,d\xi^{(n)}_\Delta(s_1,\dots,s_n).
 \end{align*}
Hence, the measure $\xi^{(n)}_\Delta$ is symmetric on $\R^n$. Therefore, for each $j=1,\dots,n$, the $j$-th coordinate projection of $\xi_\Delta^{(n)}$ is concentrated on $ \R_+$. This implies that the measure $\xi_\Delta^{(n)}$ is concentrated on $( \R_+)^n$. 

Now, fix an arbitrary $\Delta\in\mathcal B_c(X^{(n)}_{\widehat 0})$. Choose $l\in\mathbb N$ large enough so that $\Delta\subset (\Lambda_l)^{(n)}_{\widehat 0}=:\Delta'$. Then, by Lemma~\ref{huf},
$$\xi^{(n)}_\Delta(\complement (\R_+)^n) \le \xi^{(n)}_{\Delta'}(\complement (\R_+)^n)=0.$$
Here $\complement (\R_+)^n$ denotes the complement of $(\R_+)^n$.
Thus, the measure $\xi_\Delta^{(n)}$ is concentrated on $(\R_+)^n$. 
\end{proof}

\begin{lemma} \label{tyre75r86dss}
For each $n\in\mathbb N$, there exists a unique  measure $\xi^{(n)}$ on $\set_n$ which satisfies \eqref{ufru8}
for all   $\Delta\in\mathcal B_c(X^{(n)}_{\widehat0})$ and $A\in\mathcal B(( \R_+)^n)$.
\end{lemma}

\begin{proof} For each $\Delta\in\mathcal B_c(X^n)$, we define a measure $\xi^{(n)}_\Delta$ on $( \R_+)^n$ by
$$\xi^{(n)}_\Delta:=\xi^{(n)}_{\Delta\cap X_{\widehat 0}^{(n)}}.$$
(Note that $\Delta\cap X_{\widehat 0}^{(n)}\in \mathcal B_c(X^{(n)}_{\widehat0})$.) The  statement analogous to Lemma~\ref{huf} holds for $\mathcal B_c(X^n)$.
So, it suffices to prove that there exists a unique measure $\xi^{(n)}$ on $Y^n$ which satisfies
\begin{equation}\label{tye576iei}
\xi^{(n)}(\Delta\times A)=\xi_\Delta^{(n)}(A),\quad
\Delta\in\mathcal B_c(X^{n}),\ A\in\mathcal B(( \R_+)^n).\end{equation}
But this follows from the fact that, for each $\Delta\in\mathcal B_c(X^n)$, $\xi_\Delta^{(n)}$ is a measure on $(\R_+)^n$ and from Lemma~\ref{huf}, see e.g.\ Remark (3), p.~66 in \cite{Kingman}.
\end{proof}

Let us now recall a result from \cite[Corollary 1]{LM}  (see also \cite{BKKL,KK}) about existence of a unique point process with a given correlation measure. We will adopt this result to the case of the locally compact Polish space $Y=X\times\R_+$.

Let $\rho$ be a measure on $(\Gamma_0(Y),\mathcal B(\Gamma_0(Y)))$. We assume that $\rho$ satisfies the conditions (LB) and (PD) introduced below.

\begin{itemize}

\item[(LB)] {\it Local bound:} For any $\Lambda\in\mathcal B_c(X)$ and $A\in\mathcal B_c( \R_+)$, there exists a constant $\operatorname{const}_{\Lambda,A}>0$ such that
$$\rho^{(n)}((\Lambda\times A)^n\cap \set_n)\le \operatorname{const}_{\Lambda,A}^n,\quad n\in\mathbb N,$$
and for any sequence $\Lambda_k\in\mathcal B_c(X)$ such that $\Lambda_k\downarrow\varnothing$ and $A\in\mathcal B_c( \R_+)$, we have $\operatorname{const}_{\Lambda_k,A}\to0$ as $k\to\infty$.
\end{itemize}

To formulate condition (PD) we first need to give some definitions. For any measurable functions $G_1,G_2:\Gamma_0(Y)\to\R$, we define their $\star$-product as the measurable function $G_1\star G_2:
\Gamma_0(Y)\to\R$ given by
\begin{equation}\label{yufru7rdddrdsd} G_1\star G_2(\lambda):=\sum_{\substack{\lambda_1\subset\lambda,\,\lambda_2\subset\lambda\\
\lambda_1\cup\lambda_2=\lambda}} G_1(\lambda_1)G_2(\lambda_2),\quad \lambda\in\Gamma_0(Y).\end{equation}
We denote by $\mathcal S$ the class of all functions $G:\Gamma_0(Y)\to\R$ which satisfy the following assumptions:

\begin{itemize}
\item[(i)] There exists $N\in\mathbb N$ such that $G^{(n)}:=G\restriction \Gamma^{(n)}(Y)=0$ for all $n>N$.

\item[(ii)] For each $n=1,\dots,N$, the function
$G^{(n)}:=G\restriction \Gamma^{(n)}(Y)$ can be identified with  a finite linear combination of functions of the form
$$\operatorname{Sym}_n(\chi_{B_1}\otimes\dots\otimes\chi_{B_n}),$$
where for $i=1,\dots,n$ $B_i=\Lambda_i\times A_i$ with $\Lambda_i\in\mathcal B_c(X)$ and $A_i\in\mathcal B_c( \R_+)$,
$\operatorname{Sym}_n$ denotes the operator of symmetrization of a function, and
$$(\chi_{B_1}\otimes\dots\otimes\chi_{B_n})(x_1,s_1,\dots,x_n,s_n):=\chi_{B_1}(x_1,s_1)\dotsm\chi_{B_n}(x_n,s_n),$$
where $(x_1,s_1),\dots,(x_n,s_n)\in Y$ with $(x_i,s_i)\ne(x_j,s_j)$ if $i\ne j$.
\end{itemize}

It is evident that each function $G\in\mathcal S$ is bounded and integrable with respect to the measure $\rho$, and for any $G_1,G_2\in\mathcal S$, we have $G_1\star G_2\in\mathcal S$.

\begin{itemize}

\item[(PD)] {\it $\star$-positive definiteness:} For each $G\in\mathcal S$, we have
\begin{equation}\label{ctdydey67y}\int_{\Gamma_0(Y)}G\star G\,d\rho\ge0. \end{equation}
\end{itemize}

\begin{remark} Assume that $\rho$ is the correlation measure of a point process $\nu$, and assume that $\rho$ satisfies (LB). Let us briefly explain why condition (PD) must then be satisfied.
For a function $G\in\mathcal S$, we denote
$$(KG)(\gamma):=\sum_{\lambda\Subset\gamma}G(\lambda),\quad \gamma\in\Gamma(Y),$$
where $\lambda\Subset\gamma$ means that $\lambda \subset \gamma$ and $\lambda \in \Gamma_0(Y)$.
Then, by \eqref{yuf7ur7ddddytr},
$$\int_{\Gamma_0(Y)}G\,d\rho=\int_{\Gamma(Y)}KG\,d\nu.$$
Furthermore, an easy calculation shows that, for any $G_1,G_2\in\mathcal S$, we have
$$K(G_1\star G_2)=KG_1\cdot KG_2,$$
where the $\cdot$ in the above equality denotes the pointwise multiplication. Hence, in this case, formula \eqref{ctdydey67y} becomes
$$\int_{\Gamma(Y)}(KG)^2\,d\nu\ge0.$$
\end{remark}

\begin{theorem}[\!\!\cite{LM}] \label{bhftuf7urf7}
 Let $\rho$ be a measure on $(\Gamma_0(Y),\mathcal B(\Gamma_0(Y)))$ with
 \begin{equation}\label{gfrt57rd}
\rho(\Gamma^{(0)}(Y))=1\end{equation}
which satisfies the  conditions {\rm (LB)} and {\rm (PD).} Then there exists a unique point process $\nu$ in $Y$ whose correlation measure is $\rho$.
\end{theorem}

For each $n\in\mathbb N$, we  now define a measures $\rho^{(n)}$ on $\set_n$ by
\begin{equation}\label{dsrte564ew} d\rho^{(n)}(x_1,s_1,\dots,x_n,s_n):=(s_1\dotsm s_n)^{-1}\, d\xi^{(n)}(x_1,s_1,\dots,x_n,s_n).\end{equation}
Note that $\rho^{(n)}$ is a symmetric measure on $\set_n$.
We next define a measure $\rho$ on\linebreak  $(\Gamma_0(Y),\mathcal B(\Gamma_0(Y)))$ by requiring that,  for each $n\in\mathbb N$, the restriction of the measure $\rho$ to $\Gamma^{(n)}(Y)$ be  equal to $\rho^{(n)}$, i.e., for each measurable function $G:\Gamma_0(Y)\to[0,\infty]$
\begin{equation}\label{biugfyuf}
\int_{\Gamma^{(n)}(Y)}G(\lambda)\,d\rho(\lambda)=\int_{\set_n}
G(\{x_1,s_1,\dots,x_n,s_n\})\,d\rho^{(n)}(x_1,s_1,\dots,x_n,s_n).\end{equation}
For $n=0$ we define $\rho$ by \eqref{gfrt57rd}. 
A crucial part of the proof of Theorem~\ref{jkgfrt7urd} is  the following theorem.

\begin{theorem}\label{gu76rt68o}Let the measure $\rho$ on $(\Gamma_0(Y),\mathcal B(\Gamma_0(Y)))$ be defined by \eqref{gfrt57rd}--\eqref{biugfyuf}. Then
there exists a unique point process $\nu$ in $Y$ whose correlation measure is $\rho$.
\end{theorem}

In view of Theorem~\ref{bhftuf7urf7}, it suffices to prove that $\rho$ satisfies (LB) and (PD). We split the proof into several lemmas.

\begin{lemma} The measure $\rho$ defined by \eqref{gfrt57rd}--\eqref{biugfyuf} satisfies {\rm (LB).}
\end{lemma}

\begin{proof} Let $A\in\mathcal B_c( \R_+)$ and let $C:=\sup_{s\in A}s^{-1}$. Then we see by
\eqref{dsrte564ew} that
\begin{equation}\label{tyrde65e}
\rho^{(n)}((\Lambda\times A)^n\cap \set_n)\le C^n\xi^{(n)}((\Lambda\times \R_+)^n\cap \set_n)\end{equation}
for each $\Lambda\in\mathcal B_c(X)$.
By \eqref{ude7re7uvt}, \eqref{tyde6ue6iut}, Lemma~\ref{tyre75r86dss},
\begin{align}
\xi^{(n)}((\Lambda\times \R_+)^n\cap \set_n)&=\xi^{(n)}_{\Lambda^{(n)}_{\widehat 0}}(( \R_+)^n)\notag\\
&=\xi_{0,\dots,0}^{\Lambda^{(n)}_{\widehat 0}}\notag\\
&=\frac1{n!}\,M_{1,\dots,1}(\Lambda^{(n)}_{\widehat 0})\notag\\
&=\frac1{n!}M^{(n)}(\Lambda^{(n)}_{\widehat 0}).\label{re7u0ouy08}
\end{align}
Condition (LB) now follows from \eqref{tyrde65e} and \eqref{re7u0ouy08}, and condition (C2). \end{proof}

 In order to prove (PD), we rewrite this condition for non-symmetrized product of spaces. We denote
$$\Phi(Y):=\bigcup_{n=0}^\infty \Phi^{(n)}(Y),$$
where the set $\Phi^{(0)}(Y)$ contains just one element, and for $n\in\mathbb N$, $\Phi^{(n)}(Y):=\set_n$.
We define a $\sigma$-algebra $\mathcal B(\Phi(Y))$ on $\Phi(Y)$ so that, for each $n=0,1,2,\dots$, $\Phi^{(n)}(Y)\in\mathcal B(\Phi(Y))$  and for each $n\in\mathbb N$, the restriction of $\mathcal B(\Phi(Y))$ to $\Phi^{(n)}(Y)$ coincides with $\mathcal B(\set_n)$. We can  treat $\rho$ as a measure on $\Phi(Y)$, so that $\rho(\Phi^{(0)}(Y))=1$ and, for $n\in\mathbb N$, the restriction of $\rho$ to $\Phi^{(n)}(Y)$ is $\rho^{(n)}$.  We call a function $G:\Phi(Y)\to\R$ symmetric if, for each $n\in\mathbb N$, the restriction of $G$ to $\Phi^{(n)}(Y)$ is a symmetric function.
Clearly, each function $G$ on $\Gamma_0(Y)$  identifies   a symmetric function on $\Phi(Y)$, for which we preserve the notation $G$. Furthermore, for an integrable function $G$, we then have
$\int_{\Gamma_0(Y)}G\,d\rho=\int_{\Phi(Y)}G\,d\rho$.

Let $m,n\in\mathbb N$. We denote by $\operatorname{Pair}(m,n)$ the collection of all subsets $\varkappa$ of the set
$$\{1,2,\dots,m\}\times\{m+1,m+2,\dots,m+n\}$$
such that, if $(\alpha_i,\beta_i),(\alpha_j,\beta_j)\in\varkappa$ and $(\alpha_i,\beta_i)\ne(\alpha_j,\beta_j)$, then $\alpha_i\ne\alpha_j$ and $\beta_i\ne\beta_j$. By definition, an empty set  is an element of $\operatorname{Pair}(m,n)$.  For $\varkappa\in\operatorname{Pair}(m,n)$, we denote by
$|\varkappa|$ the number of elements of the set $\varkappa$. In words, this means that we build $|\varkappa|$ pairs between elements in $\{1,2,\dots,m\}$ and in $\{m+1,m+2,\dots,m+n\}$. Each element can be member of only one pair.

Let $G_1^{(m)}:\set_m\to\R$, $G_2^{(n)}:\set_n\to\R$, and let $\varkappa=\{(\alpha_i,\beta_i)\}\in \operatorname{Pair}(m,n)$.
We define a function $(G_1^{(m)}\otimes G_2^{(n)})_\varkappa:\set_{m+n-|\varkappa|}\to\R$ as follows. Relabel so that
$$\beta_1<\beta_2<\dots<\beta_{|\varkappa|}.$$
Then $(G_1^{(m)}\otimes G_2^{(n)})_\varkappa (y_1,\dots,y_{m+n-|\varkappa|})$ is defined as follows.
Take
$$ G_1^{(m)}(y_1,\dots,y_m)G_2^{(n)}(z_{m+1},\dots,z_{m+n}). $$
For each $i\in\{1,\dots,|\varkappa|\}$, replace the variable $z_{\beta_i}$ with $y_{\alpha_i}$. After this, the variables $z_j$ with $j\in\{m+1,\dots,m+n\}\setminus\{\beta_1,\dots,\beta_{|\varkappa|}\}$ (these are the remaining $z_j$)  are consecutively set to the values $y_{m+1},y_{m+2},\dots, y_{m+n-|\varkappa|}$.
Here, $y_l:=(x_l,s_l)$. In words, this means that $G_1^{(m)}$ and $G_2^{(n)}$ share some of the $y_i$ variables whose indices and positions in the arguments of $G_1^{(m)}$ and $G_2^{(n)}$ are described by the pairs in $\varkappa$. 

For example, for $m=3$, $n=4$, $\varkappa=\{(3,5),(2,6)\}$,
we have
$$ (G_1^{(3)}\otimes G_2^{(4)})_\varkappa(y_1,y_2,y_3,y_4,y_5)=G_1^{(3)}(y_1,y_2,y_3)G_2^{(4)}(y_4,y_3,y_2,y_5),\quad (y_1,y_2,y_3,y_4,y_5)\in \set_5.$$

Let us interpret  $G_1^{(m)}:\set_m\to\R$ and $G_2^{(n)}:\set_n\to\R$ as functions defined on $\Phi(Y)$ which vanish outside $\Phi^{(m)}(Y)$ and $\Phi^{(n)}(Y)$, respectively. We then define a function
$$G_1^{(m)}\diamond G_2^{(n)}:\Phi(Y)\to\R$$
by
\begin{equation}\label{vtyftuykrf}
G_1^{(m)}\diamond G_2^{(n)}:=\sum_{\varkappa\in \operatorname{Pair}(m,n)}\frac{(m+n-|\varkappa|)!}{m!\,n!}
(G_1^{(m)}\otimes G_2^{(n)})_\varkappa.
\end{equation}
In the above formula, each $(G_1^{(m)}\otimes G_2^{(n)})_\varkappa$ is also treated as a function on $\Phi(Y)$.

Note that a function $G_1^{(0)}:\Phi^{(0)}(Y)\to\R$ is just a real number and we  set, for each function $G_2:\Phi(Y)\to\R$,
\begin{equation}\label{tyre756e7}
G_1^{(0)}\diamond G_2=G_2\diamond G_1^{(0)}:=G_1^{(0)}\cdot G_2.
\end{equation}
Extending formulas \eqref{vtyftuykrf}, \eqref{tyre756e7} by linearity, we define, for any functions
$G_1,G_2:\Phi(Y)\to\R$, their  $\diamond$-product $G_1\diamond G_2$ as a function on $\Phi(Y)$.

\begin{lemma} \label{gfdtyeri7r6t} Assume that $G_1$ and $G_2$ are symmetric functions on $\Phi(Y)$ which vanish outside the set $\bigcup_{n=0}^N\Phi^{(n)}(Y)$ for some $N\in\mathbb N$. Then
$$\int_{\Phi(Y)}G_1\star G_2\,d\rho=\int_{\Phi(Y)}G_1\diamond G_2\,d\rho,$$
provided the integrals in the above formulas make sense.
\end{lemma}

\begin{proof}
It suffices to consider the case where $G_1=G_1^{(m)}:\set_m\to\R$, $G_2=G_2^{(n)}:\set_n\to\R$ for some $m,n\in\mathbb N$. Using \eqref{yufru7rdddrdsd},  we have
\begin{multline*}
\int_{\Phi(Y)}G_1^{(m)}\star G_2^{(n)}\,d\rho\\
=\sum_{k=0}^{m\wedge n}
\sum_{\substack{(\theta_1,\theta_2,\theta_3)\in\mathscr P_3(m+n-k)\\
|\theta_1|=m-k,\, |\theta_2|=k,\, |\theta_3|=n-k}}
\int_{\set_{m+n-k}} G_1^{(m)}(y_{\theta_1},y_{\theta_2})G_2^{(n)}(y_{\theta_2},y_{\theta_3})\,d\rho^{(m+n-k)}(y_1,\dots,y_{m+n-k}).
\end{multline*}
Here $\mathscr P_3(m+m-k)$ denotes the set of all ordered partitions $(\theta_1,\theta_2,\theta_3)$ of the set $\{1,\dots,m+n-k\}$ into three parts, $|\theta_i|$ denotes the number of elements in block $\theta_i$, and, for  block $\theta_i=\{r_1,r_2,\dots,r_{|\theta_i|}\}$,
$y_{\theta_i}$ denotes $y_{r_1},y_{r_2},\dots,y_{r_{|\theta_i|}}$.
Evidently, the set $\mathscr P_3(m+n-k)$ contains $\frac{(m+n-k)!}{(m-k)!\,(n-k)!\,k!}$ elements $(\theta_1,\theta_2,\theta_3)$ such that $|\theta_1|=m-k$, $|\theta_2|=k$, $|\theta_3|=n-k$. Hence
\begin{multline}
\int_{\Phi(Y)}G_1^{(m)}\star G_2^{(n)}\,d\rho
=
\sum_{k=0}^{m\wedge n}\frac{(m+n-k)!}{(m-k)!\,(n-k)!\,k!}\\
\times \int_{\set_{m+n-k}}G_1^{(m)}(x_1,\dots,x_m)G_2^{(n)}(x_{m-k+1},\dots,x_{m+n-k})\,d\rho^{(m+n-k)}(x_1,\dots,x_{m+n-k}).\label{fyt76r76re7y}
\end{multline}
On the other hand, by \eqref{vtyftuykrf},
$$
\int_{\Phi(Y)}G_1^{(m)}\diamond G_2^{(n)}\,d\rho
=\sum_{k=0}^{m\wedge n}\frac{(m+n-k)!}{m!\,n!}\sum_{\substack{\varkappa\in\operatorname{Pair}(m,n)\\|\varkappa|=k}}
\int_{\set_{m+n-k}}(G_1^{(m)}\otimes G_2^{(n)})_\varkappa\,d\rho^{(m+n-k)}.$$
An easy combinatoric argument shows that there are
$$\frac{m!}{(m-k)!\,k!}\times\frac{n!}{(n-k)!\,k!}\times k!=\frac{m!\,n!}{(m-k)!\, (n-k)!\,k!}
 $$ elements $\varkappa\in\operatorname{Pair}(m,n)$ such that $|\varkappa|=k$. Hence
 \begin{align}
&\int_{\Phi(Y)}G_1^{(m)}\diamond G_2^{(n)}\,d\rho
=\sum_{k=0}^{m\wedge n}\frac{(m+n-k)!}{m!\,n!} \times
\frac{m!\,n!}{(m-k)!\, (n-k)!\,k!} \notag\\
&\times \int_{\set_{m+n-k}}G_1^{(m)}(x_1,\dots,x_m)G_2^{(n)}(x_{m-k+1},\dots,x_{m+n-k})\,d\rho^{(m+n-k)}(x_1,\dots,x_{m+n-k}).\label{tye65e7i4ero789}
 \end{align}
 By \eqref{fyt76r76re7y} and \eqref{tye65e7i4ero789} the lemma follows.
 \end{proof}

 We denote
 $$\Psi(X):=\bigcup_{n=0}^\infty\Psi^{(n)}(X),$$
 where the set $\Psi^{(0)}(X)$ contains one element, and for $n\in\mathbb N$, $\Psi^{(n)}(X):=X^n$.
 Analogously to $\mathcal B(\Phi(Y))$, we define the $\sigma$-algebra $\mathcal B(\Psi(X))$. We next define a measure $M$ on $(\Psi(X),\mathcal B(\Psi(X)))$ so that $M(\Psi^{(0)}(X)):=M^{(0)}=1$ and, for $n\in\mathbb N$, the restriction of $M$ to $\Psi^{(n)}(X)$ is $M^{(n)}$. For any functions $F_1^{(m)}$ and $F_2^{(n)}$ on $\Psi^{(m)}(X)$ and $\Psi^{(n)}(X)$, respectively, their tensor product $F_1^{(m)}\otimes F_2^{(n)}$ is a function on $\Psi^{(m+n)}(X)$. (In the case where either $m$ or $n$ is equal to zero, the tensor product becomes a usual product.) Extending the tensor product by linearity, we define, for any functions $F_1$ and $F_2$ on $\Psi(X)$, their tensor product $F_1\otimes F_2$ as a function on $\Psi(X)$.

The following lemma shows that the measure $M$ on $\Psi(X)$ is $\otimes$-positive definite.

\begin{lemma}\label{utft7u7r}
Assume that a function $F$ on $\Psi(X)$ vanishes outside a set $\bigcup_{n=0}^N \Psi^{(n)}(X)$ for some $N\in\mathbb N$. Assume that the function $F\otimes F$ is integrable with respect to $M$. Then
\begin{equation}\label{yitu8r8fy7de}
 \int_{\Psi(X)}F\otimes F\,dM\ge0. \end{equation}
\end{lemma}

\begin{proof}
The result immediately follows from
$$\int_{\mathbb M(X)}\langle\eta^{\otimes n},F^{(n)}\rangle\,d\mu(\eta)=\int_{X^n}F^{(n)}\,dM^{(n)}.$$
\end{proof}

Let a function $g^{(n)}:X^{(n)}_{\widehat 0}\to\R$ be
 bounded, measurable, and having support from $\mathcal B_c(X^{(n)}_{\widehat 0})$. For $i_1,\dots,i_n\in\mathbb N$, we set
 \begin{equation}\label{dkjcisgv}G^{(n)}(x_1,s_1,\dots,x_n,s_n):=g^{(n)}(x_1,\dots,x_n)s_1^{i_1}\dotsm s_n^{i_n},\quad (x_1,s_1,\dots,x_n,s_n)\in \set_n.\end{equation}
 We extend the function $g^{(n)}$ by zero to the whole space $X^n$. We define a function $\mathcal R_{i_1,\dots,i_n}g^{(n)}:X^{i_1+\dots+i_n}\to\R$ by using formula \eqref{jhacvc}. We denote $G^{(0)}:\Psi^{(0)}(X)\to\R$
 \begin{equation}\label{jhzvcgv}\mathcal KG^{(n)}:=\frac1{n!}\, \mathcal R_{i_1,\dots,i_n}g^{(n)}.\end{equation}

 We denote by $\mathcal Q$ the class of all functions on $\Phi(Y)$ which are finite sums of functions of form \eqref{dkjcisgv}. Extending $\mathcal K$ by linearity, we define, for each $G\in\mathcal Q$, $\mathcal K G$ as a function on $\Psi(X)$.

\begin{lemma}\label{ygf7r76}
For each $G\in\mathcal Q$, we have
\begin{equation}\label{jhiuohgi9}
\int_{\Phi(Y)} G\,d\rho=\int_{\Psi(X)}\mathcal KG\,dM.\end{equation}
\end{lemma}

\begin{proof}  Let $\Delta\in\mathcal B_c(X^{(n)}_{\widehat 0})$ and let $G^{(n)}$ be given by \eqref{dkjcisgv} with $g^{(n)}=\chi_\Delta$. By Lemma~\ref{tyre75r86dss} and
formulas \eqref{ude7re7uvt}, \eqref{tyde6ue6iut},  \eqref{dsrte564ew}, and \eqref{jhzvcgv},
\begin{align*}
\int_{\set_n}G^{(n)}\,d\rho^{(n)}&=\int_{\set_n}\chi_\Delta(x_1,\dots,x_n)s_1^{i_1}\dotsm s_n^{i_n}\,d\rho^{(n)}(x_1,s_1,\dots,x_n,s_n)\\
&=\int_{\set_n}\chi_\Delta(x_1,\dots,x_n)s_1^{i_1-1}\dotsm s_n^{i_n-1}\,d\xi^{(n)}(x_1,s_1,\dots,x_n,s_n)\\
&=\int_{( \R_+)^n}s_1^{i_1-1}\dotsm s_n^{i_n-1}\,d\xi^{(n)}_\Delta(s_1,\dots,s_n)\\
&=\xi^\Delta_{i_1-1,\dots,i_n-1}\\
&=\frac1{n!}\, M_{i_1,\dots,i_n}(\Delta)\\
&=\int_{X^{i_1+\dots+i_n}}\frac1{n!}\,\mathcal R_{i_1,\dots,i_n}\chi_\Delta \,dM^{(i_1+\dots+i_n)}\\
&= \int_{\Psi(X)} \mathcal K G^{(n)}\,dM.
\end{align*}
From here it easily follows by linearity and approximation that formula \eqref{jhiuohgi9} holds for each $G\in\mathcal Q$.
\end{proof}

\begin{lemma}\label{bhgyuft}
For each $G\in\mathcal Q$,
$$\int_{\Phi(Y)}G\diamond G\,d\rho\ge0.$$
\end{lemma}

\begin{proof}  Let  functions $g_1^{(m)}:X_{\widehat 0}^{(m)}\to\R$ and $g_2^{(n)}:X_{\widehat 0}^{(n)}\to\R$  be bounded, measurable, and having support from $\mathcal B_c(X^{(m)}_{\widehat 0})$ and $\mathcal B_c(X^{(n)}_{\widehat 0})$, respectively. Let $i_1,\dots,i_m,j_1,\dots,j_n\in\mathbb N$.
Let
\begin{align*}
G_1^{(m)}(x_1,s_1,\dots,x_m,s_m):&=g_1^{(m)}(x_1,\dots,x_m)s_1^{i_1}\dotsm s_n^{i_m},\quad (x_1,s_1,\dots,x_m,s_m)\in \set_m\\
G_2^{(n)}(x_1,s_1,\dots,x_n,s_n):&=g_2^{(n)}(x_1,\dots,x_n)s_1^{j_1}\dotsm s_n^{j_n},\quad (x_1,s_1,\dots,x_n,s_n)\in \set_n.
\end{align*}
Then, by \eqref{jhacvc} and \eqref{jhzvcgv},
\begin{align}
&(\mathcal KG_1^{(m)}\otimes \mathcal K G_2^{(n)})(x_1,\dots,x_{i_1+\dots+i_m+j_1+\dots+j_n})\notag\\
&=\frac1{m!\,n!}\,(\mathcal R_{i_1,\dots,i_m}g_1^{(m)}\otimes\mathcal R_{j_1,\dots,j_n}g_2^{(n)})(x_1,\dots,x_{i_1+\dots+i_m+j_1+\dots+j_n})\notag\\
&=\frac1{m!\,n!}\, g_1^{(m)}(x_1,x_{i_1+1},\dots,x_{i_1+\dots+i_{m-1}+1})\notag\\
&\quad\times g_2^{(n)}(x_{i_1+\dots+i_m+1},x_{i_1+\dots+i_m+j_1+1},\dots,
x_{i_1+\dots+i_m+j_1+\dots+j_{n-1}+1})\notag\\
&\quad\times\mathcal I_{i_1,\dots,i_m}(x_1,\dots,x_{i_1+\dots+i_m})\mathcal I_{j_1,\dots,j_{n}}(x_{i_1+\dots+i_m+1},\dots,x_{i_1+\dots+i_m+j_1+\dots+j_n}).\label{hfgiuf}
\end{align}
Define for $(x_1,x_{i_1+1},\dots,x_{i_1+\dots+i_{m-1}+1})$ and
$$(x_{i_1+\dots+i_m+1},x_{i_1+\dots+i_m+j_1+1},\dots,
x_{i_1+\dots+i_m+j_1+\dots+j_{n-1}+1})$$ the number $\alpha_1$ as the lowest index $j$ such that there exists a
$$\beta_1 \in \{
i_1+\dots+i_m+1,\,i_1+\dots+i_m+j_1+1,\ldots,
 i_1+\dots+i_m+j_1+\dots+j_{n-1}+1\}$$ with $x_j =x_{j'}$. Define $(\alpha_i,\beta_i)_i $ for $i >1$ analogously. In this way one produces a $\varkappa \in \operatorname{Pair}(m,n)$. Then $\left(\mathcal I_{i_1,\dots,i_m}\otimes \mathcal I_{j_1,\dots,j_{n}}\right)_\varkappa$ is of the form
 $\mathcal I_{l_1,\dots,l_{m+n-k}}$ for appropriate $l_1, \ldots, l_k$ and $k =|\varkappa |$.
 By \eqref{vtyftuykrf}, \eqref{dkjcisgv}--\eqref{hfgiuf} and recalling that the measure $M$ is symmetric on each $\Psi^{(k)}(X)$,
$$\int_{\Psi(X)}\mathcal KG_1^{(m)}\otimes \mathcal K G_2^{(n)}\,dM=\int_{\Psi(X)}\mathcal K(G_1^{(m)}\diamond G_2^{(n)})\,dM. $$
Hence, for any $G_1,G_2\in\mathcal Q$,
\begin{equation}\label{utye67e4}
\int_{\Psi(X)}\mathcal KG_1\otimes \mathcal K G_2\,dM=\int_{\Psi(X)}\mathcal K(G_1\diamond G_2)\,dM. \end{equation}
(Note that $G_1\diamond G_2\in\mathcal Q$.) Hence, by Lemma~\ref{utft7u7r} and \eqref{utye67e4}, for each $G\in\mathcal Q$
$$ \int_{\Psi(X)}\mathcal K(G\diamond G)\,dM\ge0.$$
Now the result follows from Lemma~\ref{ygf7r76}.
\end{proof}

Next we extend the result of Lemma~\ref{bhgyuft} to a more general class of functions $G$ by approximation.

\begin{lemma}\label{gufry7f}
Let $\Lambda\in\mathcal B_c(X)$.
Let  a function $G:\Phi(Y)\to\R$ be of the form
\begin{equation}\label{111}
G=G^{(0)}+\sum_{j=1}^J G_j^{(n_j)},\end{equation}
 where $G^{(0)}:\Phi^{(0)}(Y)\to\R$, $J\in\mathbb N$, and each  function $G_j^{(n_j)}:\Phi^{(n_j)}(Y)\to\R$ is of the form
\begin{equation}\label{222}
G_j^{(n_j)}(x_1,s_1,\dots,x_{n_j}s_{n_j})= g_j^{(n_j)}(x_1,\dots,x_{n_j})f_j^{(n_j)}(s_1,\dots,s_{n_j})s_1\dotsm s_{n_j}.
\end{equation}
 Here $n_j\in\mathbb N$,  the functions $g_j^{(n_j)}$ and $f_j^{(n_j)}$ are measurable and bounded and each function  $g_j^{(n_j)}$ vanishes outside the set  $\Lambda^{(n_j)}_{\widehat0}$. Then
\begin{equation}\label{yr75i6e6ie}
\int_{\Phi(Y)} G\diamond G\,d\rho\ge0.\end{equation}
\end{lemma}

\begin{proof}
Let
$N:=\max\{n_1,n_2,\dots,n_J\}$.
For each $n\in\{1,2,\dots,N\}$,  we define a measure $\zeta_{n,N}$ on $( \R_+)^n$ by
\begin{equation}\label{hjfuyftfuft}
\zeta_{n,N}:=\sum_{i=n}^{2N}P_n\xi^{(i)}_{\Delta_i}.\end{equation}
Here
 $\Delta_i:=\Lambda^{(i)}_{\widehat0}$ and  $P_n\xi^{(i)}_{\Delta_i}$ denotes the projection of the (symmetric) measure $\xi^{(i)}_{\Delta_i}$ onto its first $n$ coordinates. Note that $\zeta_{n,N}$ is a symmetric measure on $( \R_+)^n$. We next define a  measure $Z_{n,N}$ on $( \R_+)^n$ by
\begin{equation}\label{gddrdrdfyt}
dZ_{n,N}(s_1,\dots,s_n):=d\zeta_{n,N}(s_1,\dots,s_n)\sum_{A\in\mathscr P(n)}\prod_{j\in A}s_j.\end{equation}
Here $\mathscr P(n)$ denotes the power set of $\{1,\dots,n\}$ and $\prod_{j\in\varnothing}s_j:=1$. Clearly, $Z_{n,N}$
is also a symmetric measure.
By \eqref{tyde6ue6iut}, \eqref{hjfuyftfuft}, and \eqref{gddrdrdfyt}, the moments of the measure $Z_{n,N}$ are given by
$$\int_{( \R_+)^n}s_1^{i_1}\dotsm s_n^{i_n}\,dZ_{n,N}(s_1,\dots,s_n)=\sum_{i=n}^{2N}\sum_{A\in\mathscr P(n)}
\xi^{\Delta_i}_{i_1+\chi_A(1),\dots,i_n+\chi_A(n),0\dots,0},\quad (i_1,\dots,i_n)\in\mathbb Z_+^n.$$
 Hence, by \eqref{it68r},
 \begin{equation}\label{gufttyf}
 \int_{( \R_+)^n}s_1^{i_1}\dotsm s_n^{i_n}\,dZ_{n,N}(s_1,\dots,s_n)\le (2N-n-1)2^n(i_1+\dots+i_n+n+2N)!\,C_\Lambda^{i_1+\dots+i_n+n+2N}.
 \end{equation}
By  \eqref{gufttyf} and \cite[Chapter 5, Subsec.~2.1, Examples 2.1, 2.2]{BK}, the set of polynomials is dense in $L^2(( \R_+)^n,dZ_{n,N})$.

 For each $j=1,\dots,J$, we clearly have $f_j^{(n_j)}\in L^2(( \R_+)^{n_j},dZ_{n_j,N})$. Hence, there exists a sequence of polynomials $(p_{j,k}^{(n_j)})_{k=1}^\infty$ such that
\begin{equation}\label{tyew5w5w} p_{j,k}^{(n_j)}\to f_j^{(n_j)}\text{ in }L^2(( \R_+)^{n_j},dZ_{n_j,N})\quad \text{as }k\to\infty.\end{equation}
Set $G_k:=G^{(0)}+\sum_{j=1}^J G_{j,k}^{(n_j)}$, where
$$ G_{j,k}^{(n_j)}(x_1,s_1,\dots,x_{n_j}s_{n_j}):=g_j^{(n_j)}(x_1,\dots,x_{n_j}) p_{j,k}^{(n_j)}(s_1,\dots,s_{n_j})s_1\dotsm s_{n_j}.$$
We then have $G_k\in\mathcal Q$ for each $k\in\mathbb N$.
By Lemma~\ref{bhgyuft},
\begin{equation}\label{gfyufyufyxfdR}\int_{\Phi(Y)}G_k\diamond G_k\,d\rho\ge0,\quad k\in\mathbb N.\end{equation}

We claim that
\begin{equation}\label{gtydyrdey6}\int_{\Phi(Y)}G_k\diamond G_k\,d\rho\to \int_{\Phi(Y)}G\diamond G\,d\rho
\quad\text{as }k\to\infty.\end{equation}
Indeed,  let us fix any $i,j\in\{1,\dots,J\}$  and any $\varkappa\in\operatorname{Pair}(n_i,n_j)$
with $|\varkappa|=l$,
and prove that
\begin{equation}\label{tyde6ed}
\int_{\set_{n_i+n_j-l}}(G^{(n_i)}_{i,k}\otimes G^{(n_j)}_{j,k})_\varkappa\,d\rho^{(n_i+n_j-l)}\to \int_{\set_{n_i+n_j-l}}(G_i^{(n_i)}\otimes G_j^{(n_j)})_\varkappa\,d\rho^{(n_i+n_j-l)}
\quad\text{as }k\to\infty.\end{equation}
For simplicity of notation, let us assume that $\varkappa$ is of the form
$$\{(n_i-l+1,n_i+1),\, (n_i-l+2,n_i+2),\, (n_i-l+3,n_i+3)\dots, (n_i,n_i+l)\}. $$
Then
\begin{align}
&\int_{\set_{n_i+n_j-l}}(G^{(n_i)}_{i,k}\otimes G^{(n_j)}_{j,k})_\varkappa\,d\rho^{(n_i+n_j-l)}\notag\\
&= \int_{\set_{n_i+n_j-l}}g_i^{(n_i)}(x_1,\dots,x_{n_i}) p_{i,k}^{(n_i)}(s_1,\dots,s_{n_i})\notag\\
&\quad\times
g_j^{(n_j)}(x_{n_i-l+1},x_{n_i-l+2},\dots,x_{n_i+n_j-l}) p_{j,k}^{(n_j)}(s_{n_i-l+1},s_{n_i-l+2},\dots,s_{n_i+n_j-l})\notag\\
&\quad\times s_{n_i-l+1}s_{n_i-l+2}\dotsm s_{n_i}
\,d\xi^{(n_i+n_j-l)}(x_1,s_1,\dots,x_{n_i+n_j-l}, s_{n_i+n_j-l}).\label{vtyrd6re456s}
\end{align}
Hence, there exists $C>0$ such that
\begin{align}
&\bigg|\int_{\set_{n_i+n_j-l}}(G^{(n_i)}_{i,k}\otimes G^{(n_j)}_{j,k})_\varkappa\,d\rho^{(n_i+n_j-l)}
- \int_{\set_{n_i+n_j-l}}(G^{(n_i)}_{i}\otimes G^{(n_j)}_{j,k})_\varkappa\,d\rho^{(n_i+n_j-l)}\bigg|
\notag\\
&\le \int_{\set_{n_i+n_j-l}}
\big|g_i^{(n_i)}(x_1,\dots,x_{n_i})g_j^{(n_j)}(x_{n_i-l+1},x_{n_i-l+2},\dots,x_{n_i+n_j-l})\big|\notag\\
&\quad\times
|p_{i,k}^{(n_i)}(s_1,\dots,s_{n_i})
-f_{i}^{(n_i)}(s_1,\dots,s_{n_i})|\notag\\
&\quad\times
|p_{j,k}^{(n_j)}(s_{n_i-l+1},s_{n_i-l+2},\dots,s_{n_i+n_j-l})|\notag\\
&\quad\times s_{n_i-l+1}s_{n_i-l+2}\dotsm s_{n_i}
\,d\xi^{(n_i+n_j-l)}(x_1,s_1,\dots,x_{n_i+n_j-l}, s_{n_i+n_j-l})\notag\\
&\le C\int_{\set_{n_i+n_j-l}}\chi_
{\Lambda^{(n_i+n_j-l)}_{\widehat 0}}(x_1,\dots,x_{n_i+n_j-l})\notag\\
&\quad\times
|p_{i,k}^{(n_i)}(s_1,\dots,s_{n_i})
-f_{i}^{(n_i)}(s_1,\dots,s_{n_i})|\notag\\
&\quad\times
|p_{j,k}^{(n_j)}(s_{n_i-l+1},s_{n_i-l+2},\dots,s_{n_i+n_j-l})|\notag\\
&\quad\times s_{n_i-l+1}s_{n_i-l+2}\dotsm s_{n_i}
\,d\xi^{(n_i+n_j-l)}(x_1,s_1,\dots,x_{n_i+n_j-l}, s_{n_i+n_j-l})\notag\\
&\le C
\bigg(\int_{\set_{n_i+n_j-l}}\chi_
{\Lambda^{(n_i+n_j-l)}_{\widehat 0}}(x_1,\dots,x_{n_i+n_j-l})\notag\\
&\quad\times
|p_{i,k}^{(n_i)}(s_1,\dots,s_{n_i})
-f_{i}^{(n_i)}(s_1,\dots,s_{n_i})|^2\notag\\
&\quad\times s_{n_i-l+1}s_{n_i-l+2}\dotsm s_{n_i}
\,d\xi^{(n_i+n_j-l)}(x_1,s_1,\dots,x_{n_i+n_j-l}, s_{n_i+n_j-l})
\bigg)^{1/2}\notag\\
&\quad\times
\bigg(\int_{\set_{n_i+n_j-l}}\chi_
{\Lambda^{(n_i+n_j-l)}_{\widehat 0}}(x_1,\dots,x_{n_i+n_j-l})\notag\\
&\quad\times
|p_{j,k}^{(n_j)}(s_{n_i-l+1},s_{n_i-l+2},\dots,s_{n_i+n_j-l})|^2\notag\\
&\quad\times s_{n_i-l+1}s_{n_i-l+2}\dotsm s_{n_i}
\,d\xi^{(n_i+n_j-l)}(x_1,s_1,\dots,x_{n_i+n_j-l}, s_{n_i+n_j-l})
\bigg)^{1/2}\notag\\
&\le C\,\| p_{i,k}^{(n_i)}-f_i^{(n_i)}\|_{L^2(( \R_+)^{n_i},dZ_{n_i,N})}\, \|p_{j,k}^{(n_j)}\|_{L^2(( \R_+)^{n_j},dZ_{n_j,N})}\to0\ \text{as }k\to\infty,\label{6drwqwqaqaa}
\end{align}
where we used the Cauchy inequality and \eqref{tyew5w5w}. Analogously,
\begin{equation}\label{tyed66ire}
\bigg|\int_{\set_{n_i+n_j-l}}(G^{(n_i)}_{i}\otimes G^{(n_j)}_{j,k})_\varkappa\,d\rho^{(n_i+n_j-l)}
- \int_{\set_{n_i+n_j-l}}(G^{(n_i)}_{i}\otimes G^{(n_j)}_{j})_\varkappa\,d\rho^{(n_i+n_j-l)}\bigg|\to0\ \text{as }k\to\infty.
\end{equation}
By \eqref{6drwqwqaqaa} and \eqref{tyed66ire},  formula \eqref{tyde6ed} follows.
Formula \eqref{gtydyrdey6} follows from \eqref{tyde6ed}. Now, the lemma follows from \eqref{gfyufyufyxfdR} and \eqref{gtydyrdey6}. \end{proof}

\begin{proof}[Proof of Theorem~\ref{gu76rt68o}.]
As a special case of Lemma~\ref{gufry7f}, formula \eqref{yr75i6e6ie} holds for each function $G\in\mathcal S$. Hence, by Lemma~\ref{gfdtyeri7r6t}, the measure $\rho$ satisfies condition (PD). Thus, Theorem~\ref{gu76rt68o} is proven.
\end{proof}


Since the correlation measure $\rho$ of the point process $\nu$ from Theorem~\ref{gu76rt68o} is concentrated on $\Phi(Y)$, the point process $\nu$ is concentrated on $\Gamma_p(Y)$, the set of pinpointing configurations in $Y$,
see\ e.g.\  \cite[Corollary 1]{LM}. 
Recalling formula \eqref{ddddydyyd}, one sees that for  each $\Lambda\in\mathcal B_c(X)$,
\begin{align}\int_{\Gamma_p(Y)}\mathfrak M_\Lambda\,d\nu&=\int_{\Gamma_p(Y)}\sum_{(x,s)\in\gamma}\chi_\Lambda(x)s\,d\nu(\gamma) \notag\\
&=\int_Y \chi_\Lambda(x)s \,d\rho^{(1)}(x,s)\notag\\
&=\int_Y \chi_\Lambda(x) \,d\xi^{(1)}(x,s)<\infty.\label{gf765re6i5e4i}\end{align}
Hence, $\mathfrak M_\Lambda<\infty$ $\nu$-a.s., and therefore $\nu(\Gamma_{pf}(Y))=1$, cf.\ \eqref{fdstrset6s} for the definition of $\Gamma_{pf}(Y)$. Recall the bijective mapping $\mathcal E:\mathbb K(X)\to\Gamma_{pf}(Y)$. As  already discussed in Section~\ref{5tew5w}, the inverse mapping $\mathcal E^{-1}$ is measurable. So we can define a probability measure $\mu'$ on $\mathbb K(X)$ as the pushforward of $\nu$ under $\mathcal E^{-1}$. Thus,  to finish the proof of Theorem~\ref{jkgfrt7urd}, it suffices to show that $\mu=\mu'$.

Let $\Lambda\in\mathcal B_c(X)$. Recall that, for any
$i_1,\dots,i_k\in\mathbb N$, $k\in\mathbb N$,
$$\int_{\set_k}\chi_{\Lambda^{(k)}_{\widehat 0}}(x_1,\dots,x_k)s_1^{i_1}\dotsm s_k^{i_k}\,d\rho^{(k)}(x_1,s_1,\dots,x_k,s_k)<\infty.$$
Hence, using the definition of a correlation measure (analogously as in \eqref{gf765re6i5e4i}),
we easily see that, for each $n\in\mathbb N$,
$$\int_{\Gamma_{pf}(Y)}\bigg(\sum_{(x,s)\in\gamma}\chi_\Lambda(x)s\bigg)^n\,d\nu(\gamma)<\infty.$$
 Therefore, for each $n\in\mathbb N$,
$$\int_{\mathbb K(X)}\eta(\Lambda)^n\,d\mu'(\eta)<\infty.$$
Here $\eta(\Lambda):=\langle\eta,\chi_\Lambda\rangle$, i.e., the $\eta$-measure of $\Lambda$. Hence,  $\mu'$ has finite moments. We denote by $(M^{(n)}_{\mu'})_{n=0}^\infty$ the moment sequence of the random discrete measure $\mu'$. By Theorem~\ref{ufut7fruvfguqfd} and the construction of the measure $\rho$, it follows that
\begin{equation}\label{yude64ewsa}
M'_{i_1,\dots,i_n}=M_{i_1,\dots,i_n}\,,\quad i_1,\dots,i_n\in\mathbb N,\ n\in\mathbb N,\end{equation}
where the measures $M'_{i_1,\dots,i_n}$ are defined analogously to $M_{i_1,\dots,i_n}$, by  starting with the moment sequence $(M^{(n)}_{\mu'})_{n=0}^\infty$\,, rather than $(M^{(n)})_{n=0}^\infty$. By virtue of \eqref{yude64ewsa}, the moment sequence $(M^{(n)}_{\mu'})_{n=0}^\infty$ coincides with the moment sequence $(M^{(n)})_{n=0}^\infty$.

Now, fix any sets $\Lambda_1,\dots,\Lambda_n\in \mathcal B_c(X)$. For any $i_1,\dots,i_n\in\mathbb Z_+$, we get
\begin{align}
&\int_{\mathbb K(X)}\eta(\Lambda_1)^{i_1}\dotsm
\eta(\Lambda_n)^{i_n}\,d\mu'(\eta) \notag\\
&\quad =
\int_{X^{i_1+\dots+i_n}}
\big(
\chi_{\Lambda_1}^{\otimes i_1}\otimes \dots\otimes \chi_{\Lambda_n}^{\otimes i_n}\big)(x_1,\dots,x_{i_1+\dots+i_n})
 \,dM^{(i_1+\dots+i_n)}(x_1,\dots,x_{i_1+\dots+i_n}).
 \notag\\ &\quad = \int_{\mathbb M(X)}\eta(\Lambda_1)^{i_1}\dotsm
\eta(\Lambda_n)^{i_n}\,d\mu(\eta).\label{futr7ei543eews}
\end{align}
By (C1), \eqref{futr7ei543eews}, and the Carleman criterion, the joint distribution of the random variables $\eta(\Lambda_1),\dots,\eta(\Lambda_n)$ under $\mu'$ coincides with the joint distribution of the random variables $\eta(\Lambda_1),\dots,\eta(\Lambda_n)$ under $\mu$. But it is well known (see e.g.\ \cite{Kallenberg}) that  $\mathcal B(\mathbb M(X))$ coincides with  the minimal $\sigma$-algebra on $\mathbb M(X)$ with respect to which each function $\eta\mapsto \eta(\Lambda)$ with $\Lambda\in\mathcal B_c(X)$, is measurable. Therefore, we indeed get the equality $\mu=\mu'$.
\end{proof}

\section{Moment problem on $\mathbb K(X)$ \label{secmoment}}

As a consequence of our results, we will now present a solution of the moment problem on $\mathbb K(X)$. Consider a sequence $(M^{(n)})_{n=0}^\infty$, where $M^{(0)}=1$ and for each $n\in\mathbb N$, $M^{(n)}\in\mathbb M(X^n)$ is symmetric.
Analogously to the proof of Theorem~\ref{jkgfrt7urd}, we define the measure $M$ on $\Psi(X)$. Denote by $\mathscr F$ the space of all measurable, bounded functions $F:\Psi(X)\to\R$ such that $F$ vanishes outside a set
$\Psi^{(0)}(X)\cup\left(\bigcup_{n=1}^N \Lambda^n\right)$ where $N\in\mathbb N$ and $\Lambda\in\mathcal B_c(X)$.
We will say that the sequence $(M^{(n)})_{n=0}^\infty$ is {\it positive definite\/} if,  for each $F\in\mathscr F$, \eqref{yitu8r8fy7de} holds.
Clearly, if $(M^{(n)})_{n=0}^\infty$ is the moment sequence of a random measure $\mu$, then it is positive definite.

\begin{corollary}\label{sec4cor1} Consider a sequence $(M^{(n)})_{n=0}^\infty$, where $M^{(0)}=1$ and for each $n\in\mathbb N$, $M^{(n)}\in\mathbb M(X^n)$ is symmetric. Assume that $(M^{(n)})_{n=0}^\infty$ satisfies conditions {\rm (C1)} and {\rm (C2)}. Then $(M^{(n)})_{n=0}^\infty$ is the moment sequence of a random discrete measure on $X$ if and only if $(M^{(n)})_{n=0}^\infty$ is positive definite and satisfies conditions (i) and (ii) of Theorem~\ref{jkgfrt7urd}.
\end{corollary}

\begin{proof}The result immediately follows from  Theorem~\ref{jkgfrt7urd} and its proof because the existence of $\mu$ was only used in Lemma~\ref{utft7u7r}. The assertion of this lemma is nothing else but the positive definiteness of $(M^{(n)})_{n=0}^\infty$.
\end{proof}

We also  obtain a characterization of point processes in terms of their moments.

\begin{corollary}\label{hfyrd6edhgyhg}
{\rm (i)} Let $\mu$ be a random measure on $X$, i.e., a probability measure on\linebreak $(\mathbb M(X),\mathcal B(\mathbb M(X)))$. Assume that $\mu$ has finite moments,
and let $(M^{(n)})_{n=0}^\infty$ be its moment sequence.
Further assume that conditions {\rm (C1)} and {\rm (C2)} are
satisfied.
Then $\mu$ is a simple point process, i.e., $\mu(\Gamma(X))=1$, if and only if, for any $n\in\mathbb N$ and any $i_1,\dots,i_n\in\mathbb N$, we have $M_{i_1,\dots,i_n}=M_{1,\dots,1}$, i.e., for each $\Delta\in\mathcal B(X^{(n)}_{\widehat 0})$,
\begin{equation}\label{fufrrvhghgh7dr} M_{i_1,\dots,i_n}(\Delta)=M_{1,\dots,1}(\Delta),\quad i_1,\dots,i_n\in\mathbb N.\end{equation}
In the latter case, the correlation measure $\rho$ of $\mu$ is given by
\begin{equation}\label{ffyuftydydey6} \rho^{(n)}(\Delta)=\frac1{n!}\, M^{(n)}(\Delta),\quad \Delta\in\mathcal B(X^{(n)}_{\widehat 0}),\end{equation}
where $\rho^{(n)}$ is the restriction of $\rho$ to $\Gamma^{(n)}(X)$, $\rho^{(n)}$ being identified with a measure on $X^{(n)}_{\widehat 0}$.

\noindent {\rm (ii)} Consider a sequence $(M^{(n)})_{n=0}^\infty$, where $M^{(0)}=1$ and for each $n\in\mathbb N$, $M^{(n)}\in\mathbb M(X^n)$ is symmetric. Assume that $(M^{(n)})_{n=0}^\infty$ satisfies conditions {\rm (C1)} and {\rm (C2)}. Then $(M^{(n)})_{n=0}^\infty$ is the moment sequence of a simple point process in $X$ if and only if $(M^{(n)})_{n=0}^\infty$ is positive definite and \eqref{fufrrvhghgh7dr} holds.
\end{corollary}

\begin{proof} As easily seen, it suffices to prove only part (i).
Assume that $\mu$ is a point process in $X$. Hence, $\mu$ is a random discrete measure on $X$. The corresponding point process $\nu=\mathcal E(\mu)$ is concentrated on
$$\Gamma(X\times\{1\})=\big\{\{(x,1)\}_{x\in\gamma}\mid\gamma\in\Gamma(X)\big\}.$$
Hence, $\Gamma(X\times\{1\})$ can  naturally be identified with $\Gamma(X)$, and under this identification we get $\mu=\nu$.
Furthermore, the correlation measure $\rho$ of $\mu$ coincides with the correlation measure of $\nu$, provided  we have identified
$\Gamma_0(X)$ with $\Gamma_0(X\times\{1\})$. Now, formulas \eqref{fufrrvhghgh7dr}, \eqref{ffyuftydydey6} follow from
Theorem~\ref{ufut7fruvfguqfd}.

Next, assume that $\mu$ is a random measure which satisfies \eqref{fufrrvhghgh7dr}. Hence, for any $n\in\mathbb N$ and $\Delta\in \mathcal B_c(X^{(n)}_{\widehat 0})$, we get
$$\xi_{i_1,\dots,i_n}^\Delta=\xi^\Delta_{0,\dots,0}\,,\quad i_1,\dots,i_n\in\mathbb  Z_+.$$
In other words, for each $n\in\mathbb N$ and $\Delta\in\mathcal B_c(X^{(n)}_{\widehat0})$, the moment sequence $\xi^\Delta_{\mathbf i}$ is constant and thus the measure $\xi^{(n)}_\Delta$ is concentrated at one point, $(1,\dots ,1)$. Hence, conditions (i) and (ii) Theorem~\ref{jkgfrt7urd} are satisfied, and so $\mu$ is a random discrete measure.
  Consequently, by \eqref{ufru8} and \eqref{uyr76or8o}, the measure $\rho^{(n)}$ is concentrated  on the set
$$\big\{
(x_1,1,\dots,x_n,1)\mid(x_1,\dots,x_n)\in X^{(n)}_{\widehat 0}
\big\}.$$ Therefore, the point process $\nu=\mathcal E(\mu)$ is concentrated on
$\Gamma(X\times\{1\})$. Hence, $\mu$ is a point process in $X$.
\end{proof}

  Let us now assume that $X=\mathbb R^d$, or more generally, that $X$ is a connected $C^\infty$ Riemannian manifold. Let $\mathcal D(X):=C_0^\infty(X)$ be the  space of smooth, compactly supported, real-valued functions on $X$, equipped with the nuclear space topology, see e.g.\ \cite{BK} for detail. Let $\mathcal D'(X)$ be its dual space, and let $\mathscr C(\mathcal D'(X))$
  be the cylinder $\sigma$-algebra on it. Note that $\mathbb M(X)\subset \mathcal D'(X)$ and the trace $\sigma$-algebra of
 $\mathscr C(\mathcal D'(X))$ on $\mathbb M(X)$ coincides with $\mathcal B(\mathbb M(X))$.

   For $\omega\in\mathcal D'(X)$ and $\varphi\in\mathcal D(X)$, we denote by $\langle\omega,\varphi\rangle$ their dual pairing.
    Following \cite{BKKL}, we inductively define  Wick polynomials on $\mathcal D'(X)$ by
  \begin{gather}
 \langle{:}\omega{:},\varphi\rangle:=\langle\omega,\varphi\rangle,\quad\varphi\in\mathcal D(X)\notag\\
 \langle{:}\omega^{\otimes n}{:},\varphi_1\otimes\dots\otimes\varphi_n\rangle:=\frac1{n^2}\bigg[\sum_{i=1}^{n}
 \langle\omega,\varphi_i\rangle\langle{:}\omega^{\otimes(n-1)}{:},\varphi_1\otimes
 \dots\otimes\check\varphi_i\otimes\dots\otimes\varphi_n\rangle\notag\\
 \text{}-2\sum_{1\le i<j\le n} \langle\omega,\varphi_i\rangle\langle{:}\omega^{\otimes(n-1)}{:},\varphi_1\otimes
 \dots\otimes(\varphi_j\varphi_i)\otimes\dots\otimes\check\varphi_j\otimes\dots\otimes\varphi_n\rangle\bigg],\notag\\ \varphi_1,\dots,\varphi_n\in\mathcal D(X),\ n\ge2,\label{ftrse5w}
    \end{gather}
  where $\check \varphi_i$ denotes that the factor $\varphi_i$ is absent in the tensor product.

Let $\mu$ be a probability measure on $(\mathcal D'(X),\mathscr C(\mathcal D'(X)))$ which has finite moments. For each $n\in\mathbb N$, we consider the function \begin{align}
\left( \mathcal D'(X) \right) ^{\otimes n} & \rightarrow \mathbb{R} \nonumber \\
\varphi_1 \otimes \dots \otimes \varphi_n & \mapsto \int_{\mathcal D'(X)}\langle{:}\omega^{\otimes n}{:},\varphi_1\otimes\dots\otimes \varphi_n\rangle\,d\mu(\omega). \label{defcorgen}
\end{align}
These functions form a proper generalization of the correlation measure of a point process, which in turn are a generalisation of the classical correlation functions of statistical mechanics. Hence, we call  them \emph{generalized correlation functions of the measure $\mu$}. The above results can we rewritten in terms of conditions on the generalized correlation functions.

\begin{corollary}\label{huigtt}
Assume that, for each $n\in\mathbb N$, the generalized correlation function defined in (\ref{defcorgen}) associated to a measure $\mu$ on $\mathcal D'(X)$ 
can be represented via a (positive) measure $\rho^{(n)}$ on $(X_{\widehat0}^{(n)},\mathcal B(X_{\widehat0}^{(n)}))$, that is, for any $\varphi_1,\dots,\varphi_n\in\mathcal D(X)$,
\begin{equation}\label{bhfuf}
\int_{\mathcal D'(X)}\langle{:}\omega^{\otimes n}{:},\varphi_1\otimes\dots\otimes \varphi_n\rangle\,d\mu(\omega)=
\int_{X^{(n)}_{\widehat 0}}\varphi_1(x_1)\dotsm\varphi_n(x_n)\,d\rho^{(n)}(x_1,\dots,x_n).\end{equation}
Furthermore, assume that the measures $\rho^{(n)}$ satisfy condition {\rm (C2)} in the sense that  $M^{(n)}$ is replaced with $\rho^{(n)}$ in the formulation of {\rm (C2)}. Then, $\mu$ is a point process, i.e., $\mu(\Gamma(X))=1$.
\end{corollary}

\begin{proof} Using \eqref{ftrse5w}, one can easily derive by induction a representation of a monomial $\langle\omega,\varphi_1\rangle\dotsm \langle\omega,\varphi_n\rangle$ through Wick polynomials.
This formula and \eqref{bhfuf} imply that, for each $n\in\mathbb N$, there exists a (positive) measure $M^{(n)}$ on $X^n$ such that
$$\int_{\mathcal D'(X)}\langle\omega,\varphi_1\rangle\dotsm \langle\omega,\varphi_n\rangle\,d\mu(\omega)= \int_{X^n}\varphi_1(x_1)\dotsm\varphi_n(x_n)\,dM^{(n)}(x_1,\dots,x_n).$$
Furthermore,  formulas \eqref{fufrrvhghgh7dr}, \eqref{ffyuftydydey6} hold,  because each summand in the representation of a monomial through Wick polynomials corresponds to a particular sub-diagonal $X^{(n)}_\pi$ of $X^n$. (We leave details of these calculations to the interested reader.)

By the assumption of the corollary, the sequence $(M^{(n)})_{n=0}^\infty$ with $M^{(0)}=1$ satisfies (C2). Furthermore, (C2) and  \eqref{fufrrvhghgh7dr} easily imply that $(M^{(n)})_{n=0}^\infty$ satisfies (C1). Since $(M^{(n)})_{n=0}^\infty$  is the  moment sequence of a probability measure, it is positive definite. Hence, the statement follows from Corollary \ref{hfyrd6edhgyhg}, (ii).\end{proof}

\begin{remark}
In fact, Corollary~\ref{huigtt} is essentially already contained in \cite{BKKL} and \cite[Corollary 1]{LM}, though not presented as an independent result. If we do not assume {\it a priori\/} the existence of a measure $\mu$, then we have additionally to assume that the generalised correlation functions have to fulfil the condition (PD).
Note that Theorem~\ref{bhftuf7urf7}, taken from \cite{LM}, was used in order to obtain the point process in $Y$ (Theorem~\ref{gu76rt68o}),  which in turn, was used to construct the random discrete measure on $X$. Hence, it is not surprising that we get a comparable result in the special case where instead of a random discrete measure on $X$, one actually wants to characterize a point process in $X$.
\end{remark}

\section*{Acknowledgements}
We  are grateful to the referee for their  careful reading of the manuscript and a number of helpful suggestions for improvement in the article.

The authors   acknowledge the financial support of the SFB~701 ``Spectral structures and topological methods in mathematics'' (Bielefeld University) and the Research Group
``Stochastic Dynamics: Mathematical Theory and Applications'' (Center for Interdisciplinary Research, Bielefeld University). The authors would like to thank Ilya~Molchanov for fruitful discussions.


\begin{thebibliography}{99}

\bibitem{A} Akhiezer, N.I.:
The classical moment problem and some related questions in analysis.
 Hafner Publishing Co., New York, 1965.

\bibitem{AB} Aldous, D.J.,  Barlow, M.T.:
On countable dense random sets. Seminar on Probability, XV (Univ. Strasbourg, Strasbourg, 1979/1980), pp. 311--327,
Lecture Notes in Math., 850, Springer, Berlin--New York, 1981.

\bibitem{Bauer}  Bauer, H.: Measure and integration theory.  Walter de Gruyter \&\ Co., Berlin, 2001.

\bibitem{BK} Berezansky, Y.M., Kondratiev, Y.G.: Spectral methods in infinite-dimensional analysis. Vol. 2,  Kluwer Academic Publishers, Dordrecht, 1995.


\bibitem{BKKL}  Berezansky, Y.M., Kondratiev, Y.G., Kuna, T.,  Lytvynov, E.: On a spectral representation for correlation measures in configuration space analysis. Methods Funct. Anal. Topology 5 (1999), no. 4, 87--100.

\bibitem{DVJ1}  Daley, D. J., Vere-Jones, D.: An introduction to the theory of point processes. Vol. I. Elementary theory and methods. Second edition.   Springer-Verlag, New York, 2003.

\bibitem{DVJ2} Daley, D. J., Vere-Jones, D.: An introduction to the theory of point processes. Vol. II. General theory and structure. Second edition.  Springer, New York, 2008.



 \bibitem{VGG1}    Gel'fand, I.M., Graev, M.I., Vershik, A.M.: Models of representations of current groups.   Representations of Lie groups and Lie algebras (Budapest, 1971), 121–179, Akad. Kiad{\'o}, Budapest, 1985.

\bibitem{HKPR} Hagedorn, D., Kondratiev, Y., Pasurek, T.,  R\"okner, M.: Gibbs states over the cone of discrete measures. J. Funct. Anal. 264 (2013),  2550--2583.



\bibitem{Kallenberg}  Kallenberg, O.: Random measures. Fourth edition. Akademie-Verlag, Berlin; Academic Press, London, 1986.

\bibitem{Kendall} Kendall, W.S.: Stationary countable dense random sets. Adv. in Appl. Probab. 32 (2000),  86--100.

\bibitem{Kingman} Kingman, J.F.C.:
Completely random measures.
Pacific J. Math. 21 (1967), 59--78.

\bibitem{KK}  Kondratiev, Y.G., Kuna, T.: Harmonic analysis on configuration space. I. General theory. Infin. Dimens. Anal. Quantum Probab. Relat. Top. 5 (2002),  201--233.

\bibitem{Lenard} Lenard, A.: Correlation functions and the uniqueness of the state
in classical statistical mechanics. Comm. Math. Phys.  30 (1973), 35--44.

\bibitem{LM} Lytvynov, E, Mei, L.: On the correlation measure of a family of commuting Hermitian operators with applications to particle densities of the quasi-free representations of the CAR and CCR. J. Funct. Anal. 245 (2007),  62--88.

\bibitem{RW} Rota, G.-C., Wallstrom, T.:
Stochastic integrals: a combinatorial approach.
Ann. Probab. 25 (1997), 1257--1283.

\bibitem{S} Shifrin, S. N.:
Infinite-dimensional smooth symmetric power moment problem for nuclear spaces.
S. N. Shifrin
Ukrainskii Matematicheskii Zhurnal,  28 (1976) 793--802.


\bibitem{TsVY} Tsilevich, N, Vershik, A,  Yor, M.: An infinite-dimensional analogue of the Lebesgue measure and distinguished properties of the gamma process. J. Funct. Anal. 185 (2001), 274--296.


\bibitem{VGG2}   Vershik, A.M.; Gel'fand, I.M., Graev, M.I.: Commutative model of the representation of the group of flows $\mathrm{SL}(2,\mathbf R)^X$ connected with a unipotent subgroup.  Funct. Anal. Appl. 17 (1983), 80--82.


    \end{thebibliography}
 \end{document}